\documentclass[11pt, draft]{amsart}
\usepackage{amssymb, amstext, amscd, amsmath}
\usepackage{mathtools, xypic, color, dsfont, rotating, paralist}
\usepackage[shortlabels]{enumitem}

\makeatletter
\def\@cite#1#2{{\m@th\upshape\bfseries%
[{#1\if@tempswa{\m@th\upshape\mdseries, #2}\fi}]}}
\makeatother
\theoremstyle{plain}
\newtheorem{theorem}{Theorem}[section]
\newtheorem{corollary}[theorem]{Corollary}
\newtheorem{proposition}[theorem]{Proposition}
\newtheorem{lemma}[theorem]{Lemma}

\theoremstyle{definition}
\newtheorem{definition}[theorem]{Definition}
\newtheorem{example}[theorem]{Example}

\newtheorem{remark}[theorem]{Remark}

\newtheorem*{acknow}{Acknowledgements}
\theoremstyle{remark}

\mathtoolsset{centercolon}
\newcommand{\bC}{{\mathds{C}}}
\newcommand{\bF}{{\mathds{F}}}

\newcommand{\bR}{{\mathds{R}}}
\newcommand{\bT}{{\mathds{T}}}
\newcommand{\bZ}{{\mathds{Z}}}
  \newcommand{\A}{{\mathcal{A}}}
  \newcommand{\B}{{\mathcal{B}}}

  \newcommand{\N}{{\mathcal{N}}}
\renewcommand{\O}{{\mathcal{O}}}

  \newcommand{\T}{{\mathcal{T}}}

\newcommand{\eps}{\varepsilon}
\renewcommand{\phi}{\varphi}
\def\si{\sigma}
\def\al{\alpha}
\def\be{\beta}
\def\La{\Lambda}
\def\la{\lambda}
\def\om{\omega}
\def\Om{\Omega}
\def\de{\delta}

\def\ga{\gamma}
\newcommand\vpi{\varphi}
\newcommand\wpi{\widetilde{\pi}}

\newcommand{\fA}{{\mathfrak{A}}}

\newcommand{\Bi}{{\mathbf{i}}}

\newcommand{\Bn}{{\mathbf{n}}}

\newcommand{\foral}{\text{ for all }}
\newcommand{\qand}{\quad\text{and}\quad}

\newcommand{\qfor}{\quad\text{for}\ }

\newcommand{\ca}{\mathrm{C}^*}

\newcommand{\ol}{\overline}
\newcommand{\wt}{\widetilde}
\newcommand{\wh}{\widehat}

\newcommand{\ad}{\operatorname{ad}}

\newcommand{\Aut}{\operatorname{Aut}}

\newcommand{\End}{\operatorname{End}}

\newcommand{\id}{{\operatorname{id}}}
\newcommand{\im}{{\operatorname{Im}}}

\newcommand{\mt}{\emptyset}

\newcommand{\re}{\operatorname{Re}}

\newcommand{\spn}{\operatorname{span}}

\newcommand{\supp}{\operatorname{supp}}

\newcommand{\sca}[1]{\left\langle#1\right\rangle} 
\newcommand{\nor}[1]{\left\Vert #1\right\Vert} 
\newcommand{\bo}[1]{\mathbf{#1}} 
\newcommand{\un}[1]{\underline{#1}} 

\addtocontents{toc}{\protect\setcounter{tocdepth}{1}}

\begin{document}

\title{On Nica-Pimsner algebras of C*-dynamical systems over $\mathds{Z}_+^n$}

\author[E.T.A. Kakariadis]{Evgenios T.A. Kakariadis}
\address{School of Mathematics and Statistics\\ Newcastle University\\ Newcastle upon Tyne\\ NE1 7RU\\ UK}
\email{evgenios.kakariadis@ncl.ac.uk}

\thanks{2010 {\it  Mathematics Subject Classification.}
46L05, 46L55, 46L30, 58B34}
\thanks{{\it Key words and phrases:} Exact and nuclear C*-algebras, KMS states, Nica-Pimsner algebras.}

\maketitle

\begin{abstract}
We examine Nica-Pimsner algebras associated with semigroup actions of $\mathbb{Z}_+^n$ on a C*-algebra $A$ by $*$-endomorphisms.
We give necessary and sufficient conditions on the dynamics for exactness and nuclearity of the Nica-Pimsner algebras.
Furthermore we parameterize the KMS states at finite temperature by the tracial states on $A$.
A parametrization is also shown for KMS states at zero temperature (resp. ground states) by the tracial states on $A$ (resp. states on $A$).
Finally we give a formula for obtaining tracial states on the Nica-Pimsner algebras.
\end{abstract}

\section*{Introduction}

This project lies in the intersection of three ongoing programs on analysing C*-algebras associated with C*-dynamics.
There has been a continuous interest for a generalised C*-crossed product construction of possibly non-invertible semigroup actions.
Notable examples from the one-variable case are the works of Paschke \cite{Pas80}, Peters \cite{Pet84}, Stacey \cite{Sta93}, Murphy \cite{Mur96}, and Exel \cite{Exe03} (to mention only a few).
The key idea is to consider the quotient of a Fock algebra by an ideal of redundancies.
Our standing point of view follows Arveson's program on the C*-envelope \cite{KakPet13}, and suggests that this ideal is the \v{S}ilov ideal of a nonselfadjoint operator algebra in the sense of Arveson \cite{Arv69}.
In contrast to the case of the usual C*-crossed products, these redundancies may be very complicated to allow an in-depth exploration.
An alternate route is suggested in the joint work of the author with Davidson and Fuller \cite[Introduction]{DFK14}: dilate a possibly non-invertible semigroup action $\al \colon P \to \End(A)$ to a group action $\wt\be \colon G \to \Aut(\wt{B})$ such that the distinguished quotient related to $\al \colon \bZ_+^n \to \End(A)$ is a full corner of the usual C*-crossed product related to $\wt\be \colon G \to \Aut(\wt{B})$.
Consequently the examination of the distinguished quotient can be induced via the more exploit-able and more exploit-ed theory of usual C*-crossed products.
This pathway was met with success for several semigroup actions including Ore semigroups and spanning cones in \cite{DFK14}, and in earlier works of the author for $P = \bZ_+$ \cite{Kak11}, and of the author with Katsoulis for $P = \bF_+^n$ \cite{KakKat11}.

In the current paper we focus on the case $P = \bZ_+^n$.
This connects to the third program concerning the structure of C*-algebras of product systems.
For this paper a \emph{C*-dynamical system over $\bZ_+^n$} consists of a semigroup action $\al \colon \bZ_+^n \to \End(A)$ on a C*-algebra $A$ by $*$-endomorphisms.
One can then associate two C*-algebras that reflect the lattice order structure of $\bZ_+^n$: the Toeplitz-Nica-Pimsner algebra $\N\T(A,\al)$ and the Cuntz-Nica-Pimsner algebra $\N\O(A,\al)$ \cite[Chapter 4]{DFK14}.
These algebras arise naturally in the theory of product systems, e.g. see Fowler \cite{Fow02}, Solel \cite{Sol06}, Deaconu, Kumjian, Pask and Sims \cite{DKPS10}, Sims and Yeend \cite{SimYee10}, Carlsen, Larsen, Sims and Vittadello \cite{CLSV11}.
The existence of $\N\T(A,\al)$ is readily verified by the existence of a Fock representation.
For $\N\O(A,\al)$ the situation is very different.

Carlsen, Larsen, Sims, and Vittadello \cite{CLSV11} show that Cuntz-Nica-Pimsner algebras over $\bZ_+^n$ exist as co-universal objects that satisfy a gauge invariant uniqueness theorem.
This remarkable result leaves open the exploration of their algebraic structure.
In particular describing them as universal C*-algebras with respect to a family of generators and relations.
A concrete picture of \emph{a} C*-algebra $\N\O(A,\al)$ associated with $\al \colon \bZ_+^n \to \End(A)$ was given by the author with Davidson and Fuller \cite[Definition 4.3.16]{DFK14}.
Surprisingly this was achieved through the theory of nonselfadjoint operator algebras: $\N\O(A,\al)$ is the C*-envelope of the ``lower triangular'' part of $\N\T(A,\al)$ \cite[Theorem 4.3.7 and Corollary 4.3.18]{DFK14}.
Dilation theory was used essentially to construct a non-trivial Cuntz-Nica representation in the sense of \cite{DFK14}; then $\N\O(A,\al)$ is shown to be universal with respect to these representations.
The meeting point is then achieved by showing that $\N\O(A,\al)$ satisfies a gauge invariant uniqueness theorem \cite[Theorem 4.3.17]{DFK14}; thus $\N\O(A,\al)$ coincides with the Cuntz-Nica-Pimsner algebra of \cite{CLSV11}.

In the current continuation of \cite{DFK14} we examine further the C*-structure of $\N\T(A,\al)$ and $\N\O(A,\al)$ in terms of nuclearity, exactness and KMS states.
There is an analogy of the first part with the results of Katsura on C*-correspondences \cite[Section 7]{Kat04} from which we were strongly influenced.
Nevertheless the analysis is perplexed because of the multidimensional lattice structure of $\bZ_+^n$.
We bypass this obstacle by breaking the complexity of the algebra of $\N\O(A,\al)$ into two steps, one of which is carried out at the automorphic dilation level.
For the second part, we were intrigued by the ongoing program of Laca and Raeburn \cite{LacRae10}, as well as by the growing interest in the structure of KMS states on C*-algebras (deducting symmetry and/or phase transition breaking) from the seminal work of Bost and Connes \cite{BosCon95}.
See for example \cite{HLRS14, CDL12, ExeLac03, HaPau05, Lac98, LacNes04, LacNes11, LacRae10, LRR11, LRRW13} to mention but a few inspiring works.
For our analysis we follow the previous works of Laca and Neshveyev \cite{LacNes04, LacNes11}, Laca and Raeburn \cite{LacRae10}, and Laca, Raeburn, Ramagge and Whittaker \cite{LRRW13}.
Nevertheless we enrich our results by making use of the gauge invariant uniqueness theorems, and the automorphic dilation trick established in \cite[Chapter 4]{DFK14}.
For example we allow for the action to be supported in fewer coordinates, in contrast to \cite{HLRS14}, and we use the connection to a usual C*-crossed product to produce tracial states.
This scheme is being successfully implemented also in a sequel of the current paper which concerns C*-dynamical systems over $\bF_+^n$ by row isometries \cite{Kak14-2}.
In the course of our study, the idea of creating multivariable KMS conditions comes up naturally.
We explore it further in the appendix.

We strongly feel that our methods extend to Nica-Pimsner algebras of other types of product systems.
Furthermore we believe that they may be even used to tackle the general case, once an algebraic characterization of Nica-Pimsner algebras is achieved in total.

\subsection*{Main results}

In the first part of the paper we give necessary and sufficient conditions for exactness and nuclearity of $\N\T(A,\al)$ and $\N\O(A,\al)$.
In Theorem \ref{T: exact TNP} (resp. Theorem \ref{T: nuclearity TNP}) we show that $\N\T(A,\al)$ is exact (resp. nuclear) if and only if $A$ is exact (resp. nuclear).
In Theorem \ref{T: exact CNP} we show that $\N\O(A,\al)$ is exact if and only if $A$ is exact.
Nuclearity of $A$ suffices to imply nuclearity of $\N\O(A,\al)$, however it is not necessary even for injective systems (see \cite[Example 7.7]{Kat04} and Remark \ref{R: nuclear CNP}).
In Theorem \ref{T: nuclear CNP} we give a necessary and sufficient condition for nuclearity of $\N\O(A,\al)$ in terms of the automorphic dilation $\wt{\be} \colon \bZ^n \to \Aut(\wt{B})$.

In the second part we assume that the system is unital, and we examine the $(\si,\be)$-KMS states of $\N\T(A,\al)$ and $\N\O(A,\al)$ at inverse temperature $\be = 1/T$.
The gauge action $\{\ga_{\un{z}}\}_{\un{z} \in \bT^n}$ on the Nica-Pimsner algebras induces a non-trivial action $\si$ of $\bR$ by the automorphisms
\[
\si_t = \ga_{(\exp(\la_1it), \dots, \exp(\la_n it))} \qfor \un{\la} = (\la_1, \dots, \la_n) \in \bR^n.
\]
Let us refer here just on the results where $\la_k \neq 0$ for all $k=1, \dots, n$.
We note that the gauge invariant uniqueness theorems of \cite[Chapter 4]{DFK14} is rather helpful to tackle the case where $\la_k=0$ for some $k$ (Propositions \ref{P: d-states TNP} and \ref{P: d-states CNP}).
As expected $\N\T(A,\al)$ has a rich structure of KMS states when $\be \neq 0$.
In particular we show that there is an affine homeomorphism from the simplex of the tracial states on $A$ onto the simplex of the $(\si,\be)$-KMS states on $\N\T(A,\al)$ (Theorem \ref{T: KMS TNP}).
Moreover there is an affine homeomorphism from the simplex of the tracial states (resp. states) of $A$ onto the simplex of the KMS${}_\infty$ states (resp. ground states) of $\N\T(A,\al)$ (Propositions \ref{P: ground TNP} and \ref{P: infty TNP}).
The same characterizations pass over to $\N\O(A,\al)$ (Theorem \ref{T: KMS CNP} and Corollary \ref{C: ground states CNP}).
In particular, when $\be \neq 0$ then the results read the same for tracial states on $A$ that vanish on a specified ideal $I_{\un{1}}$ of $A$.
If the system is injective then $\N\O(A,\al)$ admits only $(\si,0)$-KMS states, i.e. tracial states.
At this point we use the dilation method of \cite[Chapter 4]{DFK14} to construct tracial states on $\N\O(A,\al)$, and consequently on $\N\T(A,\al)$.

The gauge actions on $\N\T(A,\al)$ and $\N\O(A,\al)$ lift to actions of $\bR^n$, and all of our computations hold when considering
\[
\si_{(t_1, \dots, t_n)} = \ga_{(\exp(it_1), \dots, \exp(it_n))} \qfor (t_1,\dots,t_n) \in \bR^n.
\]
This suggests the idea of carrying out the analysis at a broader multivariable context.
In order to provide a full study we explore this idea in the appendix.

Defining a multivariable KMS condition may seem to be straightforward.
However what is important is the association of a KMS condition with analytic functions on an infinite domain.
To this end we introduce a multivariable KMS condition associated to tubes (multidimensional strips) and prescribing sets.
These requirements follow directly from our analysis.
In particular the additional data on the directions of $\bR^n$ is used for estimating a uniform bound when applying the Phragm\'{e}n-Lindel\"{o}f principle coordinate-wise, in order to achieve the multivariable interplay (Theorem \ref{T: KMS char}).
This context appears to be much stronger and affects entirely the multivariable analysis.
As an example we show that such multivariable KMS states do not exist for the Toeplitz-Nica-Pimsner algebra even in the simplest case of the trivial system of $\bZ_+^2$ over $\bC$; unless of course the multivariable actions follow as reductions from the one-variable case (Example \ref{E: non-example}).

\begin{acknow}
The current paper initially started as a joint work with Ken Davidson and Adam Fuller.
Following the suggestion of Ken and Adam, it was decided for the paper to go as single-authored.
I thank Ken and Adam for their suggestion and the conversations that contributed to the preparation of this paper.
I also thank Marcelo Laca for the stimulating discussions on the theory of KMS states.

This paper is dedicated to my wonderful and forbearing soul-mate; you are the joy of my life.
\end{acknow}

\section{Preliminaries}

\subsection{Notation}

We write $\bo1, \dots, \Bi, \dots, \Bn$ for the standard generators of $\bZ_+^n$, and we simply write $\un{0} := (0,\dots,0)$ and $\un 1 := (1,\dots,1)$.
We write
\[
|\un{x}|:= \sum_{i=1}^n x_i \qand \sca{\un{x}, \un{y}} := \sum_{i=1}^n x_i y_i
\]
for all $\un{x}, \un{y} \in \bR^n$.
We denote by ``$\leq$'' the partial order in $(\bR^n, +)$ induced by $\bR_+^n$, and likewise for $\bZ^n$.
We denote \emph{the support of $\un x = (x_1,\dots,x_n) \in \bZ_+^n$} by
\[
\supp \un x := \{\Bi : x_i\ne 0 \}
\]
and we write
\[
\un x^\perp := \{\un y \in \bZ_+^n : \supp \un y \cap \supp \un x = \mt \}.
\]
A \emph{grid} is a subset of $\bR^n$, or $\bZ_+^n$, that is $\vee$-closed in $\bR^n$, or $\bZ_+^n$ respectively, where
\[
\un{x} \vee \un{y} = ((\max\{x_i,y_i\})_i) \foral \un{x}, \un{y} \in \bR^n.
\]
We will be using the grids
\[
F_{m} := \{ \un{y} \in \bZ_+^n \mid \un{y} \leq m \cdot \un{1}\}.
\]

A C*-dynamical system $\al \colon \bZ_+^n \to \End(A)$ will be called \emph{unital}/ \emph{injective}/ \emph{automorphic} if $\al_{\un{x}}$ is unital/injective/automorphic for all $\un{x} \in \bZ_+^n$.

The reader is addressed to the general result \cite[Theorem 3.1.1]{Lor97} from Loring or the usual C*-folklore technique for the construction of the universal C*-algebras.
Alternatively see \cite[Chapter 2]{DFK14}.

\subsection{The Toeplitz-Nica-Pimsner algebra}

An \emph{isometric Nica covariant representation} $V$ of $\bZ_+^n$ is a semigroup homomorphism into the semigroup of isometries on some Hilbert space $H$ such that the $V_\Bi$ doubly commute, i.e. $V_\Bi V_{\bo{j}}^* = V_{\bo{j}}^* V_\Bi$ for all $\Bi \neq \bo{j}$ \cite{Nic92} (see also \cite{DFK14}).
Following the theory of product systems we can define the following universal object.

\begin{definition}
A pair $(\pi,V)$ is called \emph{isometric Nica covariant} if:
\begin{enumerate}
\item $\pi \colon A \to \B(H)$ is a $*$-representation;
\item $V \colon \bZ_+^n \to \B(H)$ is an isometric Nica covariant representation;
\item $\pi(a)V_{\Bi} = V_{\Bi} \pi\al_\Bi(a)$ for all $a\in A$ and $\Bi = \bo1, \dots, \bo{n}$.
\end{enumerate}
The universal C*-algebra $\N\T(A,\al)$ generated by $V_{\un{x}} \pi(a)$ for $\un{x} \in \bZ_+^n$ and $a\in A$ under the isometric Nica covariant pairs $(\pi,V)$ is called the \emph{Toeplitz-Nica-Pimsner} C*-algebra.
\end{definition}

It is immediate that if $(\pi,V)$ is an isometric Nica covariant pair for $\al\colon \bZ_+^n \to \End(A)$ then
\[
\ca(\pi,V) = \ol{\spn} \{ V_{\un{x}} \pi(a) V_{\un{y}}^* \mid a\in A, \un{x}, \un{y} \in \bZ_+^n\}.
\]
Furthermore if $\al$ is unital then $1 \in A$ is the unit for $\N\T(A,\al)$.
We say that $(\pi,V)$ admits a gauge action $\{\ga_{\un{z}}\}_{\un{z} \in \bT^n}$ if
\[
\ga_{\un{z}}(\pi(a)) = \pi(a) \qand \ga_{\un{z}} (V_{\un{x}}) = z_1^{x_1} \dots z_n^{x_n} V_{\un{x}}
\]
for all $\un{z} \in \bT^n$ and $a \in A$, $\un{x} \in \bZ_+^n$.
Then we let
\[
\B_{[\un{x},\un{y}]}:=\spn\{V_{\un{w}} \pi(a) V_{\un{w}}^* \mid a\in A, \un{x} \leq \un{w} \leq \un{y}\}
\]
be \emph{the cores of the gauge action on $\ca(\pi,V)$}.
We denote the core $\B_{[\un{x},\un{x}]}$ simply by $\B_{\un{x}}$.
It is not immediate but every $\B_{[\un{x},\un{y}]}$ is a C*-subalgebra of $\ca(\pi,V)$ \cite[Corollary 4.2.8]{DFK14}.
We will also use the notation
\[
\B_{[\un{0}, \infty \cdot \bo1 + \dots + \infty \cdot \bo{d}]} = \ol{\spn} \{ V_{\un{x}} \pi(a) V_{\un{x}}^* \mid a \in A, \supp \un{x} \subseteq \{\bo1, \dots, \bo{d}\}\}.
\]
Notice that $\N\T(A,\al)$ admits a gauge action due to its universal property.

A particular isometric Nica covariant pair is given by the \emph{Fock representations}.
Let $\pi \colon A \to \B(H)$ be a $*$-representation of $A$. On $K = H \otimes \ell^2(\bZ_+^n)$, define the orbit representation $\wt{\pi} \colon A \to \B(K)$ and $V \colon \bZ_+^n \to \B(K)$ by
\begin{align*}
\wt{\pi}(a) \xi \otimes e_{\un{x}} = (\pi\al_{\un{x}}(a)\xi) \otimes e_{\un{x}}
\qand
V_{\un{y}} (\xi \otimes e_{\un{x}}) = \xi \otimes e_{\un{x} + \un{y}}
\end{align*}
for all $a \in A$, $\un{x}, \un{y} \in \bZ_+^n$ and $\xi\in H$.
In \cite[Theorem 4.2.9]{DFK14} we showed that when $\pi$ is faithful then the Fock representation induced by the pair $(\wt\pi,V)$ defines a faithful representation of $\N\T(A,\al)$.
Due to the Fock representation the ambient C*-algebra $A$ embeds isometrically inside $\N\T(A,\al)$, and therefore we will simply write $a \in \N\T(A,\al)$ for all $a \in A$.

More generally we have the following gauge invariant uniqueness theorem \cite[Theorem 4.2.11]{DFK14}:
an isometric Nica covariant pair $(\pi,V)$ defines a faithful representation of $\N\T(A,\al)$ if and only if it admits a gauge action and
\[
I_{(\pi,V)}:=\{a \in A \mid \pi(a) \in \B_{(\un{0},\infty]}\} = (0),
\]
where $\B_{(\un{0},\infty]} = \ol{\spn}\{V_{\un{x}} \pi(a) V_{\un{x}}^* \mid a \in A, \un{0} \neq \un{x}\}$.

\subsection{The Cuntz-Nica-Pimsner algebra}\label{Ss: CNP}

We require a family of ideals related to $\al \colon \bZ_+^n \to \End(A)$ that is coined in \cite{DFK14}.
For each $\un x \in \bZ_+^n\setminus\{\un{0}\}$, consider the ideal $\big( \bigcap_{\Bi\in \supp \un x} \ker\al_\Bi \big)^\perp$ and let
\[
 I_{\un x} = \bigcap_{\un y \in \un x^\perp} \al_{\un y}^{-1}
 \Big(\big( \bigcap_{\Bi\in \supp \un x } \ker\al_\Bi \big)^\perp\Big).
\]
Every $I_{\un{x}}$ is an $\al_{\un{y}}$-invariant ideal for all $\un{y} \in \un{x}^\perp$.
In particular we have that $I_{\un{0}} = \{0\}$, and $I_{\un x} = I_{\un 1} = \big( \bigcap_{\Bi=\bo1}^\Bn \ker\al_\Bi \big)^\perp$ for all $\un x \ge \un 1$.

The definition of $I_{\un{x}}$ is inspired by the following observation.
If $(\pi,V)$ is an isometric Nica covariant pair such that $\pi$ is a faithful representation of $A$, then the equation
\[
\sum_{ \un{0} \leq \un w \leq \un x} V_{\un w} \pi(a_{\un w}) V_{\un w}^* =0
\]
implies that $a_{\un{0}} \in I_{\un x}$, see \cite[Proposition 4.3.2]{DFK14}.

One has the freedom to define a family of universal objects determined by a system of generators and relations.
The interesting part in any case is to prove that the object exists and it is not trivial.
For $\al \colon \bZ_+^n \to \End(A)$ we give the following definition.

\begin{definition}\label{D: CN}
\cite[Definition 4.3.16]{DFK14}
An isometric Nica covariant pair $(\pi,U)$ is called \emph{Cuntz-Nica covariant} if
\[
\pi(a) \cdot \prod_{\Bi \in \supp \un x} (I - U_\Bi U_\Bi^*) = 0, \foral a \in I_{\un x}.
\]
The universal C*-algebra $\N\O(A,\al)$ generated by $U_{\un{x}} \pi(a)$ for $\un{x} \in \bZ_+^n$ and $a\in A$ under the Cuntz-Nica covariant pairs $(\pi,U)$ is called the \emph{Cuntz-Nica-Pimsner} C*-algebra.
\end{definition}

Evidently if $\al$ is unital then $1 \in A$ is the unit for $\N\O(A,\al)$.
A key ingredient that we used in \cite{DFK14} for proving the existence of $\N\O(A,\al)$ is a tail adding technique.
With this process we pass from a possibly non-injective system $\al \colon \bZ_+^n \to \End(A)$ to an automorphic system $\wt{\be} \colon \bZ^n \to \Aut(\wt{B})$.
Let us include here a description of this dilation.

If $\al \colon \bZ_+^n \to \End(A)$ is injective then it admits a minimal automorphic extension $\wt{\al} \colon \bZ^n \to \Aut(\wt{A})$.
Here $\wt{A}$ is the direct limit C*-algebra associated with the connecting $*$-homomorphisms $\al_{\un{x}} \colon A_{\un{y}} \to A_{\un{x} + \un{y}}$ for $A_{\un{y}} :=A$ and $\un{x}, \un{y} \in \bZ_+^n$.
The $*$-automorphism $\wt{\al}_{\un{w}}$ is defined by the diagram
\[
\xymatrix{
A_{\un{y}} \ar[r]^{\al_{\un{x}}} \ar[d]^{\al_{\un{w}}} & A_{\un{x} + \un{y}} \ar[d]^{\al_{\un{w}}} \ar[r] &\wt{A} \ar[d]^{\wt\al_{\un{w}}} \\
A_{\un{y}} \ar[r]^{\al_{\un{x}}} & A_{\un{x} + \un{y}}  \ar[r] &\wt{A}
}
\]
for every $\un{w} \in \bZ_+^n$.

Now suppose that $\al \colon \bZ_+^n \to \End(A)$ is not injective.
First suppose that $\al$ is unital.
Define $B_{\un x} := A/I_{\un x}$ for $\un x \in \bZ_+^n$ and let $q_{\un x}$ be the quotient map of $A$ onto $B_{\un x}$.
Set $B = \sum_{\un x \in \bZ_+^n}^{\oplus} B_{\un x}$.
A typical element of $B$ is denoted by
\[
b = \sum_{\un x\in\bZ_+^n} q_{\un x}(a_{\un x}) \otimes e_{\un x} ,
\]
where $a_{\un x} \in A$.
Observe that $I_{\un x}$ is invariant under $\al_\Bi$ when $x_i=0$, thus we can define the action of $\bZ_+^n$ on $B$ by
\begin{align*}
 \be_{\Bi}(q_{\un x}(a) \otimes e_{\un x} ) =
 \begin{cases}
 q_{\un x}\al_\Bi(a) \otimes e_{\un x} + q_{\un x + \Bi}(a) \otimes e_{\un x + \Bi} & \text{ if } x_i=0,\\
 q_{\un x}(a) \otimes e_{\un x + \Bi} & \text{ for } x_i\ge 1.
 \end{cases}
\end{align*}
It is clear that the compression of $\be \colon \bZ_+^n \to \End(B)$ to the $\un{0}$-th co-ordinate $A$ is the system $\al\colon \bZ_+^n \to \End(A)$.
So $\be \colon \bZ_+^n \to \End(B)$ is a dilation of $\al\colon \bZ_+^n \to \End(A)$.
Furthermore $\be \colon \bZ_+^n \to \End(B)$ is injective and it admits the minimal automorphic extension $\wt{\be} \colon \bZ^n \to \Aut(\wt{B})$.

If there is at least one $\al_\Bi$ that is not unital then form \emph{the unitization $\al^{(1)} \colon \bZ_+^n \to \End(A^{(1)})$ of the system} where $A^{(1)} = A + \bC$.
We consider $A^{(1)}$ even when $A$ has a unit, but we impose $\al^{(1)} = \al$ when the system is unital.
Let $\be^{(1)} \colon \bZ_+^n \to \End(B^{(1)})$ be the dilation of $\al^{(1)} \colon \bZ_+^n \to \End(A^{(1)})$. This is \emph{not} the unitization of $\be \colon \bZ_+^n \to \End(B)$ but the ideals $I_{\un{x}}$ turn out to be the same, i.e.
\[
B^{(1)} = \sum^{\oplus}_{\un{x} \in \bZ_+^n} B_{\un{x}}^{(1)} \; \text{ for } \; B_{\un{x}}^{(1)} = A^{(1)}/ I_{\un{x}}.
\]
Further, let $\wt{\be^{(1)}} \colon \bZ_+^n \to \End(\wt{B^{(1)}})$ be the automorphic extension.
Since $\be^{(1)} \colon \bZ_+^n \to \End(B^{(1)})$ extends $\be \colon \bZ_+^n \to \End(B)$ it follows by the gauge invariant uniqueness theorem for crossed products that $\wt{B} \rtimes_{\wt{\be}} \bZ^n$ embeds canonically and isometrically inside $\wt{B^{(1)}} \rtimes_{\wt{\be^{(1)}}} \bZ^n$.
Let $\pi \colon \wt{B^{(1)}} \to \B(H)$ be a faithful representation and let $(\wh{\pi},U)$ be the left regular pair that defines a faithful representation of $\wt{B^{(1)}} \rtimes_{\wt{\be^{(1)}}} \bZ^n$ on $H \otimes \ell^2(\bZ^n)$.
Let $1$ be the identity of the unitization $A^{(1)}$ and let $p = \wh{\pi}(1 \otimes e_{\un{0}})$.
Then the pair $(\wh{\pi}|_A, Up)$ defines a faithful representation for $\N\O(A,\al)$.
Then \cite[Theorem 4.3.7]{DFK14} and \cite[Corollary 4.3.18]{DFK14} imply that $\N\O(A,\al)$ is a full corner of $\widetilde{B}\rtimes_{\widetilde{\beta}}\bZ^n$, by the projection $p$ above).
If in particular $\al$ is injective then $\N\O(A,\al) \simeq \wt{A} \rtimes_{\wt{\al}} \bZ^n$.

The crucial remark is that once more the ambient C*-algebra $A$ embeds isometrically in $\N\O(A,\al)$.
Therefore we will simply write $a \in \N\O(A,\al)$ for all $a \in A$.
Recall that $\N\O(A,\al)$ admits a gauge action due to its universal property.
Then the gauge invariant uniqueness theorem (as in \cite[Theorem 4.3.17]{DFK14}) asserts that an isometric Cuntz-Nica covariant pair $(\pi,U)$ of $\N\O(A,\al)$ defines a faithful representation if and only if it admits a gauge action and $\pi$ is faithful.

There is a strong connection with product systems.
Cuntz-Nica-Pimsner algebras of such objects have been studied in the influential papers of Fowler \cite{Fow02}, Sims and Yeend \cite{SimYee10}, and Carlsen, Larsen, Sims and Vittadello \cite{CLSV11}.
From the analysis in \cite{DFK14} it follows that $\N\O(A,\al)$ falls inside this framework.
Among many others, the key idea in \cite{CLSV11} is to compare Cuntz-Nica-Pimsner algebras with co-universal C*-algebras subject to a gauge invariant uniqueness theorem.
However, due to the vast complexity of product systems, the defining CNP-representations in the sense of \cite[p. 568]{CLSV11} involve checking a series of properties.
A more flexible algebraic characterization, analogous to the one offered by Katsura \cite{Kat04} for C*-correspondences, is not clear at this point (and perhaps this is a price that needs to be payed).

The connection of \cite{DFK14} with \cite{CLSV11} is obtained via Dilation Theory.
In \cite[Theorem 4.3.17]{DFK14} we directly prove the gauge invariant uniqueness theorem for $\N\O(A,\al)$.
Then the combination of \cite[Theorem 4.3.17]{DFK14} with \cite[Corollary 4.14]{CLSV11} implies that the CNP-representations of \cite{CLSV11} admit the characterization of Definition \ref{D: CN} in our context.
The additional gain is that the analysis in \cite{DFK14} is carried within Arveson's Program on the C*-envelope and thus answers also a question raised in \cite[Introduction]{CLSV11}.
It seems that progress in the same spirit can help clarify further the algebraic nature of the CNP-representations for other cases of product systems.
Of course a unified approach may be set as a ultimate goal.

We conclude with a remark that follows from scattered arguments in \cite[Section 4.3 and Section 4.4]{DFK14}.

\begin{proposition}\label{P: injectivity}
Let $\al \colon \bZ_+^n \to \End(A)$ be a C*-dynamical system and let $\N\O(A,\al) \simeq \ca(\pi, U)$.
Then the following are equivalent:
\begin{enumerate}[\upshape(i)]
\item the system is injective;
\item $I_{\Bi} = A$ for all $\Bi = \bo1, \dots, \bo{n}$;
\item the $U_\Bi$ can be assumed to be unitaries for $\Bi = \bo1, \dots, \bo{n}$.
\end{enumerate}
\end{proposition}
\begin{proof}
If item (i) holds then $\ker\al_\Bi^\perp = A$ for all $\Bi = \bo{1}, \dots, \bo{n}$, hence $I_{\Bi} = A$ for all $\Bi = \bo{1}, \dots, \bo{n}$.
Consequently item (ii) holds.
Moreover by \cite[Theorem 4.3.7 and Corollary 4.3.18]{DFK14} we get that $\N\O(A,\al) \simeq \wt{A} \rtimes_{\wt{\al}} \bZ^n$ so that the $U_\Bi$ can be assumed to be unitaries.

If item (ii) holds then $\al_\Bi$ is injective, since $I_{\Bi} \subseteq \ker\al_\Bi^\perp$ for all $\Bi = \bo{1}, \dots, \bo{n}$.
Consequently item (i) holds.

Now assume that item (iii) holds and let $a \in \ker\al_\Bi$.
Then the covariant relation implies that $a U_{\Bi} = U_{\Bi} \al_\Bi(a) =0$.
However $U_\Bi$ is assumed unitary, thus $a=0$ and item (i) follows.
\end{proof}

\subsection{Exactness and Nuclearity}

The reader is addressed to \cite{BroOza08} for full details on the subject.
For the convenience of the reader and for future reference, let us state here the results that we will use.

\begin{lemma}\label{L: nuc/ex}
\begin{inparaenum}[\upshape(i)]
\item \cite[Proposition A.6]{Kat04} Suppose that the following diagram with exact rows is commutative
\begin{align*}
\xymatrix@C=4em{
0 \ar[r] & I \ar[r]^{\iota} \ar[d]^{\phi_0} & A \ar[r]^{\pi} \ar[d]^{\phi} & B \ar[r] \ar@{=}[d] & 0 \\
0 \ar[r] & I' \ar[r]^{\iota'} & A' \ar[r]^{\pi'} & B \ar[r] & 0
}
\end{align*}
and $\phi$ is injective.
Then $\phi$ is nuclear if and only if both $B$ and $\phi_0$ are nuclear.\\
\item \cite[Proposition 2]{DLRZ02}, \cite[Proposition A.13]{Kat04} Let $\ga \colon G \curvearrowright A$ be an action of a compact group $G$ on a C*-algebra $A$.
    Then $A$ is exact (resp. nuclear) if and only if the fixed point algebra $A^\ga$ is exact (resp. nuclear).
\end{inparaenum}
\end{lemma}

As an immediate consequence of the arguments of Katsura \cite{Kat04} we obtain the following lemma.

\begin{lemma}\label{L: dl nuclear}
Let the C*-algebras $B_{[n,n]} \subseteq \B(H)$ for $n \in \bZ_+$, and set $B_{[0,0]} = A$.
Define $B_{[0,n]} = \sum_{k=0}^n B_{[k,k]}$ and suppose that $B_{[1,n]} \lhd B_{[0,n]}$ (hence the $B_{[0,n]}$ are C*-algebras).
Moreover let the inductive limit C*-algebra $B_{[0,\infty]} = \ol{\cup_{n \geq 0} B_{[0,n]}}$ and its ideal $B_{(0,\infty]} = \ol{\cup_{n \geq 1} B_{[1,n]}}$.
If the embeddings
\[
A \hookrightarrow B_{[0,\infty]} \qand B_{[1,n]} \hookrightarrow B_{(0,\infty]}
\]
are nuclear for all $n \in \bZ_+$, and if there is an ideal $I \lhd A$ such that
\[
A/ I \simeq B_{[0,n]}/ B_{[1,n]} \simeq B_{[0,\infty]}/ B_{(0,\infty]}
\]
for all $n \in \bZ_+$, then $B_{[0,\infty]}$ is nuclear.
\end{lemma}
\begin{proof}
For $n=0$ we obtain the following commutative diagram
\[
\xymatrix@C=3em{
0 \ar[r] & I \ar[r] \ar[d] & A \ar[r] \ar[d] & A/I \ar[r] \ar@{=}[d] & 0 \\
0 \ar[r] & B_{(0,\infty]} \ar[r] & B_{[0,\infty]} \ar[r] & B_{[0,\infty]}/ B_{(0,\infty]} \ar[r] & 0
}
\]
which by \cite[Proposition A.6]{Kat04} asserts that $A/I \simeq B_{[0,\infty]}/ B_{(0,\infty]}$ is nuclear.
Applying \cite[Proposition A.6]{Kat04} again to the following commutative diagram
\[
\xymatrix@C=3em{
0 \ar[r] & B_{[1,n]} \ar[r] \ar[d] & B_{[0,n]} \ar[r] \ar[d] & B_{[0,n]}/B_{[1,n]} \ar[r] \ar@{=}[d] & 0 \\
0 \ar[r] & B_{(0,\infty]} \ar[r] & B_{[0,\infty]} \ar[r] & B_{[0,\infty]}/ B_{(0,\infty]} \ar[r] & 0
}
\]
shows that the embedding of $B_{[0,n]} \hookrightarrow B_{[0,\infty]}$ is nuclear.
\end{proof}

\subsection{Kubo-Martin-Schwinger states}

Let us recall some elements from the theory of KMS states.
The reader is addressed to  \cite{BraRob87,BraRob97} for full details.
Let $\si \colon \bR \to \Aut(\A)$ be an action on a C*-algebra $\A$.
For $a \in \A$ define the sequence
\[
a_m = \sqrt{\frac{m}{\pi}} \int_\bR \si_t(a) e^{-mt^2} dt.
\]
This sequence converges to $a$ and the function $t \mapsto \si_t(a_m)$ extends to the entire analytic function
\[
z \mapsto f_{m}(z) = \sqrt{\frac{m}{\pi}} \int_\bR \si_t(a) e^{-m(t-z)^2} dt.
\]
The elements $a_m$ for all $a \in \A$ then form the dense $*$-subalgebra $\A_{\textup{an}}$ of analytic elements in $\A$ \cite[Proposition 2.5.22]{BraRob87}.
Moreover $\A_{\textup{an}}$ is $\si$-invariant.
A state $\psi$ of $\A$ is called a \emph{$(\si,\be)$-KMS state} if it satisfies
\[
\psi(a b) = \psi(b \si_{i\be}(a)),
\]
for all $a,b$ in a norm-dense $\si$-invariant $*$-subalgebra of $\A_{\text{an}}$.
For $\be=0$ this amounts to tracial states.
When $\si$ is trivial, then $\A_{\textup{an}} = \A$ and the KMS property is again equivalent to $\tau$ being a tracial state.

When $\be \neq 0$ then the KMS states give rise to particular continuous functions.
More precisely for $\be>0$ let
\[
D = \{ z \in \bC \mid 0 < \im(z) < \be\}.
\]
Then a state $\psi$ is a $(\si,\be)$-KMS state if and only if for any pair $a,b \in \A$ there exists a complex function $F_{a,b}$ that is analytic on $D$ and continuous (hence bounded) on $\ol{D}$ such that
\[
F_{a,b}(t) = \psi(a \si_t(b)) \qand F_{a,b}(t + i \be) = \psi(\si_t(b)a),
\]
for all $t\in \bR$ \cite[Proposition 5.3.7]{BraRob97}.
A similar result is obtained when $\be<0$ for $D = \{ z \in \bC \mid \be < \im(z) < 0\}$.

Following \cite{LacRae10} we adopt the distinction between ground states and KMS$_\infty$ states.
A state $\psi$ of a C*-algebra $\A$ is called a \emph{ground state} if the function $z \mapsto \psi(a \si_{z}(b))$ is bounded on $\{z \in \bC \mid \text{Im}z >0\}$ for all $a,b$ inside a dense analytic subset $\A_{\text{an}}$ of $\A$.
A state $\psi$ of $\N\T(A,\al)$ is called a \emph{KMS$_\infty$ state} if it is the w*-limit of $(\si,\be)$-KMS states as $\be \longrightarrow \infty$.

When there is a multivariable action $\si \colon \bR^n \to \Aut(\A)$ one may wish to examine KMS states along an induced action $\si F \colon \bR \to \Aut(\A)$ for a continuous $F \colon \bR \to \bR^n$.
Since the action must be by automorphisms then necessarily $F(t) = (\la_1 t, \dots, \la_n t)$ for some $\la_1, \dots, \la_n \in \bR$.

\section{Exactness and Nuclearity}

\subsection{The Toeplitz-Nica-Pimsner algebra}

We begin by giving necessary and sufficient conditions for $\N\T(A,\al)$ to be exact.

\begin{theorem}\label{T: exact TNP}
Let $\al \colon \bZ_+^n \to \End(A)$ be a C*-dynamical system.
Then the following are equivalent:
\begin{enumerate}[\upshape(i)]
\item $A$ is exact;
\item the fixed point algebra $\N\T(A,\al)^\ga$ is exact;
\item $\N\T(A,\al)$ is exact.
\end{enumerate}
\end{theorem}
\begin{proof}
Obviously item (iii) implies item (i), and by \cite[Proposition A.13]{Kat04} we have the equivalence of items (ii) and (iii).
Therefore it suffices to show that exactness of $A$ implies exactness of $\N\T(A,\al)^\ga$.
To accomplish this we use the faithful Fock representation $(\wt{\pi},V)$ of $\N\T(A,\al)$ given by a faithful representation $\pi\colon A \to \B(H)$.
Recall that the fixed point algebra $\N\T(A,\al)^\ga$ can be described as the inductive limit of the C*-subalgebras
\[
\B_{[0,k \cdot \un{1}]} = \spn \{V_{\un{y}} \wpi(a) V_{\un{y}}^* \mid a\in A, \un{0} \leq \un{y} \leq k \cdot \un{1} \}.
\]
The proof will be completed once we show that every $\B_{[0,k \cdot \un{1}]}$ is exact.

\medskip

\noindent {\bf Claim.} The C*-algebra $\B_{[\un{0},k\cdot \bo1]} = \spn\{ V_{l \cdot \bo1} \wpi(a) V_{l \cdot \bo1}^* \mid a \in A, 0 \leq l \leq k\}$ is exact for all $k \in \bZ_+$.

\smallskip

\noindent {\bf Proof of the Claim.} By assumption $A$ is exact, and note that $\B_{[\un{0},\bo{1}]} = A + \B_{\bo1}$, where $\B_{\bo1}$ is an ideal of $\B_{[\un{0}, \bo1]}$.
Moreover we obtain
\[
\B_{\bo1} = \spn \{ V_\bo1 \wpi(a) V_\bo1^* \mid a \in A \} = V_\bo1 \wpi(A) V_\bo1^*,
\]
thus $\B_{\bo1}$ is exact.
If $A \bigcap \B_{\bo1} = (0)$ then $\B_{[\un{0},\bo{1}]}$ will be exact.
However, this is immediate since $\B_{\bo1} \subseteq \B_{(\un{0},\infty]}$ and $A \bigcap B_{(\un{0},\infty]} = (0)$.

Now assume that $\B_{[\un{0}, l\cdot \bo1]}$ is exact and recall that
\[
\B_{[\un{0},(l+1) \cdot \bo1]} = \B_{[\un{0}, l\cdot \bo1]} + \B_{(l+1)\cdot \bo1},
\]
where $\B_{(l+1)\cdot \bo1}$ is an ideal of $\B_{[\un{0},(l+1) \cdot \bo1]}$.
Once more we get that
\[
\B_{(l+1)\cdot \bo1} = \spn \{ V_{(l+1) \cdot \bo1} \wpi(a) V_{(l+1) \cdot \bo1}^* \mid a \in A \} = V_{(l+1) \cdot \bo1} \wpi(A) V_{(l+1) \cdot \bo1}^*,
\]
is exact.
It suffices to show that $\B_{[\un{0}, l\cdot \bo1]} \bigcap \B_{(l+1)\cdot \bo1} = (0)$.
To this end let
\[
X = V_{(l+1) \cdot \bo1} \wpi(a) V_{(l+1) \cdot \bo1}^* \in \B_{[\un{0}, l\cdot \bo1]},
\]
thus there are $a_m \in A$ such that
\[
V_{(l+1) \cdot \bo1} \wpi(a) V_{(l+1) \cdot \bo1}^* = \sum_{m=0}^l V_{m \cdot \bo1} \wpi(a_m) V_{m \cdot \bo1}^*.
\]
Then for every $\xi \in H$ we get that
\[
X \xi \otimes e_{\un{0}} = \sum_{m=0}^l V_{m \cdot \bo1} \wpi(a_m) V_{m \cdot \bo1}^* \xi \otimes e_{\un{0}} = \pi(a_0)\xi \otimes e_{\un{0}},
\]
and, on the other hand, that
\[
X \xi \otimes e_{\un{0}} = V_{(l+1) \cdot \bo1} \wpi(a) V_{(l+1) \cdot \bo1}^* \xi \otimes e_{\un{0}} =0.
\]
Therefore $a_0 =0$.
By induction on the vectors $\xi \otimes e_{m \cdot \bo1}$ for $m=0, 1, \dots, l$, we get that $X=0$, which completes the proof of the claim.

\medskip

For the general case, assume that
\[
\B_{[\un{0},\un{x}]} = \spn\{ V_{\un{y}} \wpi(a) V_{\un{y}}^* \mid a \in A, \un{0} \leq \un{y} \leq \un{x}\}
\]
is exact for $\un{x} = k \cdot \bo1 + \cdots + k \cdot (\bo{i} - \bo{1})$.
We will show that
\[
\B_{[\un{0},\un{x} + k \cdot \Bi]} = \spn\{ V_{\un{y}} \wpi(a) V_{\un{y}}^* \mid a \in A, \un{0} \leq \un{y} \leq \un{x} + k \cdot \Bi\}
\]
is exact.
Let us describe here the method that we will use.
By assumption the support of $\B_{[\un{0},\un{x}]}$ is a square on $(i-1)$-dimensions of size $k$.
Seen inside the $i$-th dimensional world it becomes a line segment (a hyperplane) of ``length'' $k$.
In order to enlarge it to a square on $i$-dimensions of size $k$ we add parallel ``rows'' of length $k$ one after the other to the direction of $\Bi$.
We use induction on the rows that we add on the $\Bi$-th direction. Thus, suppose that
\[
\B_{[\un{0},\un{x} + l \cdot \Bi]} = \spn\{ V_{\un{y}} \wpi(a) V_{\un{y}}^* \mid a \in A, \un{0} \leq \un{y} \leq \un{x} + l \cdot \Bi\}
\]
is exact for some $l \leq k-1$.
We will show that
\[
\B_{[\un{0},\un{x} + (l+1) \cdot \Bi]} = \spn\{ V_{\un{y}} \wpi(a) V_{\un{y}}^* \mid a \in A, \un{0} \leq \un{y} \leq \un{x} + (l+1) \cdot \Bi\}
\]
is exact.
If $\un{z} \in [\un{0},\un{x} + l \cdot \Bi]$ and $\un{w} \in [(l+1) \cdot \Bi, \un{x} + (l+1) \cdot \Bi]$ then
\[
\un{z} \vee \un{w} \in [(l+1) \cdot \Bi, \un{x} + (l+1) \cdot \Bi].
 \]
Therefore $\B_{[(l+1) \cdot \Bi, \un{x} + (l+1) \cdot \Bi]}$ is an ideal of $\B_{[\un{0},\un{x} + (l+1) \cdot \Bi]}$.
In particular we may write
\[
\B_{[\un{0},\un{x} + (l+1) \cdot \Bi]} = \B_{[\un{0},\un{x} + l \cdot \Bi]} + \B_{[(l+1) \cdot \Bi, \un{x} + (l+1) \cdot \Bi]},
\]
since the set
\[
\{\un{y} \in \bZ_+^n \mid 0 \leq y_1, \dots, y_{i-1} \leq k, 0\leq y_i \leq l+1, y_{i+1} = \dots = y_n =0\}
\]
decomposes into the disjoint sets
\[
\{\un{y} \in \bZ_+^n \mid 0 \leq y_1, \dots, y_{i-1} \leq k, 0 \leq y_i \leq l, y_{i+1} = \dots = y_n =0\}
\]
and
\[
\{\un{y} \in \bZ_+^n \mid 0 \leq y_1, \dots, y_{i-1} \leq k, y_i=l+1, y_{i+1} = \dots = y_n =0\}.
\]
It remains to show that
\[
\B_{[\un{0},\un{x} + l \cdot \Bi]} \bigcap \B_{[(l+1) \cdot \Bi, \un{x} + (l+1) \cdot \Bi]} = (0).
\]
However this follows in exactly the same way as in the claim since
\[
\B_{[(l+1) \cdot \Bi, \un{x} + (l+1) \cdot \Bi]}|_{H \otimes e_{\un{z}}} =0 \foral \un{z} \in [\un{0},\un{x} + l \cdot \Bi].
\]
Indeed every $\un{z} \in [\un{0},\un{x} + l \cdot \Bi]$ cannot be larger than any element inside $[(l+1) \cdot \Bi, \un{x} + (l+1) \cdot \Bi]$, and the proof is complete.
\end{proof}

Secondly we give necessary and sufficient conditions for $\N\T(A,\al)$ to be nuclear.

\begin{theorem}\label{T: nuclearity TNP}
Let $\al \colon \bZ_+^n \to \End(A)$ be a C*-dynamical system.
Then the following are equivalent:
\begin{enumerate}[\upshape(i)]
\item $A$ is nuclear;
\item the fixed point algebra $\N\T(A,\al)^\ga$ is nuclear;
\item $\N\T(A,\al)$ is nuclear.
\end{enumerate}
\end{theorem}
\begin{proof}
For the implications [(i) $\Rightarrow$ (ii) $\Leftrightarrow$ (iii)] proceed as in the proof of Theorem \ref{T: exact TNP} by replacing exactness with nuclearity.
It suffices to show that item (ii) implies item (i).
To this end note that for $X = \sum_{\un{x}=\un{0}}^{k \cdot \un{1}} V_{\un{x}} \wpi(a_{\un{x}}) V_{\un{x}}^*$ and $\xi \in H$ we obtain
\[
X\xi\otimes e_{\un{0}} = \pi(a_{\un{0}})\xi \otimes e_{\un{0}} \in H\otimes e_{\un{0}}.
\]
However the subspace $H \otimes e_{\un{0}}$ is reducing for $\N\T(A,\al)^\ga$, and hence the projection on $H \otimes e_{\un{0}}$ defines a $*$-homomorphism onto $\wpi(A)$.
Thus $A$ is nuclear, since nuclearity passes to quotients.
\end{proof}

\subsection{The Cuntz-Nica-Pimsner algebra}

Now we pass to the examination of exactness and nuclearity of $\N\O(A,\al)$.

\begin{theorem}\label{T: exact CNP}
Let $\al \colon \bZ_+^n \to \End(A)$ be a C*-dynamical system.
Then the following are equivalent:
\begin{enumerate}[\upshape(i)]
\item $A$ is exact;
\item the fixed point algebra $\N\O(A,\al)^\ga$ is exact;
\item $\N\O(A,\al)$ is exact.
\end{enumerate}
\end{theorem}
\begin{proof}
If item (i) holds then $\N\T(A,\al)$ and $\N\T(A,\al)^\ga$ are exact by Theorem \ref{T: exact TNP}.
Since $\N\O(A,\al)$ and $\N\O(A,\al)^\ga$ are quotients of $\N\T(A,\al)$ and $\N\T(A,\al)^\ga$ respectively, we obtain that items (ii) and (iii) hold.
Moreover item (ii) implies item (iii) by \cite[Proposition A.13]{Kat04}, and item (iii) implies item (i) since $A$ is represented faithfully inside $\N\O(A,\al)$.
\end{proof}

Next we focus on nuclearity of $\N\O(A,\al)$.
We isolate the injective case.
Recall that if $\al \colon \bZ_+^n \to \End(A)$ is injective then it can be extended to the automorphic system $\wt{\al} \colon \bZ^n \to \Aut(\wt{A})$. In this case we have that $\N\O(A,\al)^\ga = \wt{A}$ and $\N\O(A,\al) \simeq \wt{A} \rtimes_{\wt{\al}} \bZ^n$.

\begin{proposition}\label{P: inj nuc}
Let $\al \colon \bZ_+^n \to \End(A)$ be an injective C*-dynamical system.
Then the following are equivalent:
\begin{enumerate}[\upshape(i)]
\item the embedding $A \hookrightarrow \wt{A}$ is nuclear;
\item the fixed point algebra $\wt{A} = \N\O(A,\al)^\ga$ is nuclear;
\item $\N\O(A,\al)$ is nuclear.
\end{enumerate}
\end{proposition}
\begin{proof}
The equivalence of items (ii) and (iii) follows by \cite[Proposition A.13]{Kat04}.
Also it is immediate that item (ii) implies item (i).
In order to show that item (i) implies item (ii) it suffices to show that every embedding $A_{\un{x}} = w_{\un{x}}(A) \hookrightarrow \wt{A}$ is nuclear.
To this end suppose that the embedding $A \hookrightarrow \wt{A}$ is nuclear by an approximation
\[
\xymatrix@C=.75cm{
A \ar[rr]^{\phi_n} & & M_{k(n)} \ar[rr]^{\psi_n} & & \wt{A}
}
\]
and fix ${\un{x}} \in \bZ_+^n$.
By construction of the system $\wt{\al} \colon \bZ^n \to \Aut(\wt{A})$ we get that $\wt{\al}_{\un{x}}(A_{\un{x}}) = A_0 = A$.
Therefore the approximation
\[
\xymatrix@C=.75cm{
A_{\un{x}} \ar[rr]^{\phi_n \wt\al_{\un{x}}} & & M_{k(n)} \ar[rr]^{\wt\al_{-\un{x}} \psi_n} & & \wt{A}
}
\]
implies that the $*$-homomorphism
\[
\wt{\al}_{-{\un{x}}} \wt{\al}_{{\un{x}}}|_{A_{\un{x}}} \equiv \id_{A_{\un{x}}} \colon A_{\un{x}} \hookrightarrow \wt{A}
\]
is nuclear, and the proof is complete.
\end{proof}

\begin{remark}
The previous result holds also for injective C*-dynamical systems over spanning cones $P$ (see \cite{DFK14} for the pertinent definitions).
\end{remark}

\begin{theorem}\label{T: nuclear CNP}
Let $\al \colon \bZ_+^n \to \End(A)$ be a C*-dynamical system.
Let the C*-dynamical systems $\be \colon \bZ_+^n \to \End(B)$ and $\wt{\be} \colon \wt{B} \to \Aut(\wt{B})$ be as in Subsection \ref{Ss: CNP}.
Then the following are equivalent:
\begin{enumerate}
\item the embeddings $B_{\un{x}} \hookrightarrow \wt{B}$ are nuclear for all $\un{x} \in \bZ_+^n$;
\item the embedding $B \hookrightarrow \wt{B}$ is nuclear;
\item $\wt{B}$ is nuclear;
\item $\wt{B} \rtimes_{\wt{\be}} \bZ^n$ is nuclear;
\item $\N\O(A,\al)$ is nuclear.
\end{enumerate}
\end{theorem}
\begin{proof}
Since the C*-algebras $B_{\un{x}}$ are orthogonal ideals of $B$, and $B$ is an inductive limit of $\bigcup_{\un{x} \leq k \cdot \un{1}} B_{\un{x}}$ then the embedding $B \hookrightarrow \wt{B}$ is nuclear.
The equivalences of items (ii), (iii), and (iv) follow from Proposition \ref{P: inj nuc}.
Moreover since all the $B_{\un{x}}$ are C*-subalgebras of $\wt{B}$, item (iii) implies item (i).
The equivalence of items (iv) and (v) is immediate since $\N\O(A,\al)$ is strong Morita equivalent to $\wt{B} \rtimes_{\wt{\be}} \bZ^n \simeq \N\O(\wt{B},\wt{\be})$.
\end{proof}

\begin{remark}\label{R: nuclear CNP}
It is evident that when $A$ is nuclear then $\N\O(A,\al)$ is nuclear.
However the converse is not true.
Katsura constructs such a counterexample for the one-variable case in \cite[Example 7.7]{Kat04}.
The same construction can be extended to give a multivariable counterexample.
That is, let $A = \oplus_{n \in \bZ} A_n$ where $A_n = B$ is nuclear for $n > 0$ and $A_n = D$ is a non-nuclear C*-subalgebra of $B$ for all $n \leq 0$.
Let $\al_{\bo1}$ be the forward shift on $A$ and let $\al_{\bo2} = \dots = \al_{\bo{n}} = \id_A$.
Then the fixed point algebra $\N\O(A,\al)^\ga$ coincides with $B_{[0, \infty \cdot \bo{1}]}$ which is the direct limit of $A$ by $\al_{\bo1}$.
Thus $\N\O(A,\al)^\ga$ coincides with $\oplus_{n \in \bZ} B$.
The latter is nuclear hence $\N\O(A,\al)$ is nuclear.
However by construction $A$ is not nuclear.
\end{remark}

We give a second condition that implies the nuclearity of $\N\O(A,\al)$.

\begin{theorem}\label{T: nuclear CNP 2}
Let $\al \colon \bZ_+^n \to \End(A)$ be a C*-dynamical system.
If the embedding $A \hookrightarrow B_{[\un{0},\infty\cdot \Bi]}$ is nuclear for some $\Bi \in \{\bo{1}, \dots, \bo{n}\}$ then $\N\O(A,\al)$ is nuclear.
\end{theorem}
\begin{proof}
Let $(\pi,U)$ be a Cuntz-Nica covariant pair such that $\N\O(A,\al) \simeq \ca(\pi, U)$.
Without loss of generality we assume that $\Bi = \bo1$.
We will show that the fixed point algebra $\N\O(A,\al)^\ga$ is nuclear.
To this end we will show inductively that the C*-subalgebras
\[
\B_{[\un{0}, \infty \cdot \bo1 + \dots + \infty \cdot \bo{m}]} = \ol{\spn}\{U_{\un{x}} a U_{\un{x}}^* \mid \supp \un{x} = \{\bo1, \dots, \bo{m}\},  a \in A\}
\]
of $\ca(\pi,U)$ are nuclear.

We begin with $\bo{m}=\bo1$.
First we show that
\[
A/ I_{\bo1} \simeq \B_{[\un{0},l \cdot \bo1]}/ \B_{(\un{0}, l \cdot \bo1]} \simeq \B_{[\un{0},\infty \cdot \bo1]}/ \B_{(\un{0}, \infty \cdot \bo1]},
\]
for all $l > 0$.
Indeed note that $\B_{[\un{0},l \cdot \bo1]}/ \B_{(\un{0}, l \cdot \bo1]} \simeq A/ A \cap \B_{(\un{0}, l \cdot \bo1]}$ and recall that
\[
I_\bo1 \subseteq \B_{\bo1} \subseteq \B_{(\un{0}, l \cdot \bo1]} \subseteq \B_{(\un{0},\infty \cdot \bo1]}.
\]
Now let $a \in A \bigcap \B_{(\un{0}, l \cdot \bo1]}$ and recall that $U_\bo1 U_\bo1^*$ is an identity for $\B_{(\un{0}, l \cdot \bo1]}$.
Thus $a = a U_\bo1U_\bo1^*$ which implies that $ a(I - U_\bo1 U_\bo1^*)=0$.
Hence we get that $a \in I_\bo1$.
Therefore we obtain $A \bigcap \B_{(\un{0}, l \cdot \bo1]} = I_\bo1$.
The very same arguments give also that $A \bigcap \B_{(\un{0},\infty \cdot \bo1]} = I_\bo1$.

We aim to apply Lemma \ref{L: dl nuclear} and so we need to verify that the embeddings $\B_{(\un{0},l \cdot \bo1]} \hookrightarrow \B_{(\un{0},\infty \cdot \bo1]}$ are nuclear.
We show this inductively.
By assumption suppose that
\[
\xymatrix@C=.75cm{
A \ar[rr]^{\phi_n} & & M_{k(n)} \ar[rr]^{\psi_n} & & \B_{[\un{0},\infty \cdot \bo1]}
}
\]
is an approximation of $A \hookrightarrow \B_{[\un{0},\infty \cdot \bo1]}$.
Then
\[
\xymatrix@C=.75cm{
\B_{\bo1} = U_{\bo1} A U_{\bo1}^* \ar[rr]^{\quad \phi_n \ad_{U_\bo1^*}} & & M_{k(n)} \ar[rr]^{\ad_{U_\bo1} \psi_n} & & \B_{(\un{0},\infty \cdot \bo1]}
}
\]
is an approximation of $\B_{\bo1} = U_{\bo1} A U_{\bo1}^* \hookrightarrow \B_{(\un{0},\infty \cdot \bo1]}$.
Therefore applying \cite[Proposition A.6]{Kat04} to the following diagram
\[
\xymatrix{
0 \ar[r] & \B_{\bo1} \ar[r] \ar[d] & \B_{[0,1 \cdot \bo1]} \ar[r] \ar[d] & \B_{[0,1 \cdot \bo1]}/ \B_{\bo1} \ar[r] \ar@{=}[d] & 0 \\
0 \ar[r] & \B_{(\un{0},\infty \cdot \bo1]} \ar[r] & \B_{[\un{0},\infty \cdot \bo1]} \ar[r] & \B_{[\un{0},\infty \cdot \bo1]}/ \B_{(\un{0},\infty \cdot \bo1]} \ar[r] & 0
}
\]
yields that the embedding $\B_{[\un{0},1 \cdot \bo1]} \hookrightarrow \B_{[\un{0},\infty \cdot \bo1]}$ is nuclear.
Now assume that we have shown that the embedding $\B_{[\un{0},l \cdot \bo1]} \hookrightarrow \B_{[\un{0},\infty \cdot \bo1]}$ is nuclear by an approximation
\[
\xymatrix@C=.75cm{
\B_{[\un{0},l \cdot \bo1]} \ar[rr]^{\phi_n} & & M_{k(n)} \ar[rr]^{\psi_n} & & \B_{[\un{0},\infty \cdot \bo1]}
}.
\]
Then the approximation
\[
\xymatrix@C=.75cm{
\B_{(\un{0},(l+1) \cdot \bo1]} \ar[rr]^{\quad \phi_n \ad_{U_\bo1^*}} & & M_{k(n)} \ar[rr]^{\ad_{U_\bo1} \psi_n} & & \B_{(\un{0},\infty \cdot \bo1]}
}
\]
shows that the embedding $\B_{(\un{0},(l+1) \cdot \bo1]} \hookrightarrow \B_{(\un{0},\infty \cdot \bo1]}$ is nuclear.
Lemma \ref{L: dl nuclear} implies then that $\B_{[\un{0},\infty \cdot \bo1]}$ is nuclear.

For $\bo{m} = \bo2$ we repeat the same arguments by substituting $A$ with the C*-subalgebra $\B_{[\un{0},\infty \cdot \bo1]}$ of $\B_{[\un{0},\infty \cdot \bo1 + \infty \cdot \bo2]}$.
Note that $\B_{[\un{0},\infty \cdot \bo1]}$ is nuclear now, hence the embedding $\B_{[\un{0},\infty \cdot \bo1]} \hookrightarrow \B_{[\un{0},\infty \cdot \bo1 + \infty \cdot \bo2]}$ is nuclear.
A moment's thought suggests that $I_\bo1$ has to be substituted by the ideal
\[
\{X \in \B_{[\un{0},\infty \cdot \bo1]} \mid X \cdot (I-U_\bo2U_\bo2^*) =0\},
\]
since $U_\bo2 U_\bo2^*$ is an identity for the C*-algebras
\[
\spn\{U_{l \cdot \bo2} X U_{l \cdot \bo2}^* \mid X \in \B_{[\un{0},\infty \cdot \bo1]}, 0 \neq l \leq m\},
\]
for all $m \in \bZ_+$.
Induction then completes the proof.
\end{proof}

\section{KMS states on Nica-Pimsner algebras}\label{S: KMS}

In this section we work under the assumption that the system $\al \colon \bZ_+^n \to \End(A)$ is unital. Consequently the unit of $A$ is also a unit for $\N\T(A,\al)$ and $\N\O(A,\al)$.

\subsection{The Toeplitz-Nica-Pimsner algebra}

Let $\{\ga_{\un{z}}\}_{\un{z} \in \bT^n}$ be the gauge action on $\N\T(A,\al)$.
For a fixed $\un{\la} \in \bR^n$ let the action
\[
\si \colon \bR \to \Aut(\N\T(A,\al)): t \mapsto \ga_{(\exp(i\la_1t), \dots, \exp(i\la_nt))}.
\]
The monomials of the form $V_{\un{x}} a V_{\un{y}}^*$ span a dense $*$-subalgebra of analytic elements of $\N\T(A,\al)$ since the function
\[
\bR \to \N\T(A,\al): \un{t} \mapsto \si_t(V_{\un{x}} a V_{\un{y}}^*) = e^{i \sca{\un{x} - \un{y}, \un{\la}}t} V_{\un{x}} a V_{\un{y}}^*
\]
is analytically extended to the entire function
\[
\bC \to \N\T(A,\al) : z \mapsto e^{i\sca{\un{x} - \un{y}, \un{\la}}z} V_{\un{x}} a V_{\un{y}}^*.
\]
For $\be \in \bR$ the $(\si,\be)$-KMS condition is then translated into
\begin{align*}
\psi(V_{\un{x}} a V_{\un{y}}^* \cdot V_{\un{z}} b V_{\un{w}}^*)
& =
e^{-\sca{\un{x} - \un{y}, \un{\la}}\be} \psi(V_{\un{z}} b V_{\un{w}}^* \cdot V_{\un{x}} a V_{\un{y}}^*).
\end{align*}
As we will soon see there are no KMS states when $\la_k\be <0$ for some $k=1, \dots, n$.
On the other hand there is a characterization of the KMS states when $\la_k\be >0$ for all $k=1, \dots, n$.
We leave the case where some $\la_k$ are zeroes for later.
To simplify notation in our computations we will frequently use
\[
\un{\be} := \un{\la} \be = (\la_1\be, \dots, \la_n \be)
\]
throughout the proofs, i.e. $\be_k := \la_k \be$ for all $k=1, \dots,n$.

\begin{proposition}\label{P: cond KMS}
Let $\al \colon \bZ_+^n \to \End(A)$ be a unital C*-dynamical system, $\si \colon \bR \to \Aut(\N\T(A,\al))$ be the action related to $\un{\la} \in \bR^n$ as above, and $\be \in \bR$.

\begin{inparaenum}[\upshape(i)]
\item If $\la_k \be<0$ for some $k=1, \dots, n$, then there are no $(\si,\be)$-KMS states.

\item If $\la_k \be> 0$ for all $k=1, \dots, n$, then a state $\psi$ of $\N\T(A,\al)$ is a $(\si,\be)$-KMS state if and only if
    \[
    \psi(ab)=\psi(ba) \qand \psi(V_{\un{x}} a V_{\un{y}}^*) = \de_{\un{x}, \un{y}} e^{-\sca{\un{x}, \un{\la}}\be} \psi(a),
    \]
    for all $a,b \in A$ and $\un{x}, \un{y} \in \bZ_+^n$.
\end{inparaenum}
\end{proposition}
\begin{proof}
Since the system is unital then the unit $1 \in A$ is also the unit for $\N\T(A,\al)$.
We write $\un{\be} = \un{\la} \be = (\la_1\be, \dots, \la_n \be)$ so that $\be_k = \la_k \be$ for all $k=1, \dots, n$.

For item (i) we compute
\begin{align*}
1
=
\psi(1)
& \geq
\psi(V_{\bo{k}} 1 V_{\bo{k}}^*)
 =
\psi(1 V_{\bo{k}}^* \si_{i\be}(V_{\bo{k}}))
 =
e^{-\be_k},
\end{align*}
for all $k = 1, \dots, n$.
Hence we get that $\be_k = \la_k \be \geq 0$ for all $k=1, \dots, n$.

For item (ii) fix $\un{\la} \in \bR^n$ and $\be$ such that $\be_k = \la_k \be>0$ for all $k=1, \dots,n$.
Suppose first that $\psi$ is a $(\si,\be)$-KMS state.
Then we obtain
\[
\psi(ab) = \psi(b \si_{i\be}(a)) = \psi(ba),
\]
and that
\begin{align*}
\psi(V_{\un{x}} a V_{\un{y}}^*) = \psi(a V_{\un{y}}^* \si_{i\be}(V_{\un{x}})) = e^{-\sca{\un{x}, \un{\be}}}\psi(a V_{\un{y}}^*V_{\un{x}}).
\end{align*}
It suffices to show that $\psi(a V_{\un{y}}^*V_{\un{x}})=0$ for all $a\in A$ when $\un{x} \neq \un{y}$.
By using the Nica covariance we may assume that the monomial $a V_{\un{y}}^*V_{\un{x}}$ is written in a reduced form, i.e. $\supp \un{x} \bigcap \supp \un{y} = \emptyset$ so that $V_{\un{y}}^* V_{\un{x}} = V_{\un{x}} V_{\un{y}}^*$.
First suppose that $\un{x} \neq \un{0}$ and let an $x_k > 0$.
Since $\bo{k} \in \un{y}^\perp$ we get
\begin{align*}
\psi(a V_{\un{y}}^*V_{\un{x}})
& =
\psi(a V_{\bo{k}} V_{\un{y}}^* V_{\un{x} - \bo{k}})
=
\psi(V_{\bo{k}} \al_{\bo{k}}(a) V_{\un{y}}^* V_{\un{x} - \bo{k}}) \\
& =
\psi(\al_{\bo{k}}(a) V_{\un{y}}^* V_{\un{x} - \bo{k}} \si_{i \be}(V_{\bo{k}}))
=
e^{-\be_k} \psi(\al_{\bo{k}}(a) V_{\un{y}}^*V_{\un{x}}).
\end{align*}
Inductively we get that $\psi(a V_{\un{y}}^*V_{\un{x}}) = e^{-l \cdot\be_k} \psi(\al_{l \cdot \bo{k}}(a) V_{\un{y}}^*V_{\un{x}})$ for every $l \in \bZ_+$.
Hence we get that
\[
|\psi(a V_{\un{y}}^*V_{\un{x}})| \leq e^{-l \cdot\be_k} \nor{\al_{l \cdot \bo{k}}(a) V_{\un{y}}^*V_{\un{x}}} \leq e^{-l \cdot\be_k} \nor{a}.
\]
Taking the limit over $l$ yields $\psi(a V_{\un{y}}^*V_{\un{x}})=0$ for all $a\in A$ and $\un{y} \in \bZ_+^n$, whenever $\un{x} \neq 0$.
Now if $\un{x} = \un{0}$ then we get that $\un{y} \neq \un{0}$.
By taking adjoints and applying the KMS condition we then obtain
\[
\psi(a V_{\un{y}}^*) = \ol{\psi(V_{\un{y}} a^*)} = e^{- \sca{\un{y}, \un{\be}}} \ol{\psi(a^* V_{\un{y}})} = 0.
\]

To prove the converse of item (ii) it suffices to show that if $\psi$ satisfies the two conditions then
\begin{align*}
\psi(V_{\un{x}}aV_{\un{y}}^* \cdot V_{\un{z}} b V_{\un{w}}^*)
& =
e^{-\sca{\un{x} - \un{y}, \un{\be}}}\psi(V_{\un{z}} b V_{\un{w}}^* \cdot V_{\un{x}} a V_{\un{y}}^*),
\end{align*}
for all $\un{x}, \un{y}, \un{z}, \un{w} \in \bZ_+^n$ and $a, b \in A$.
For simplicity set $f = V_{\un{x}} a V_{\un{y}}^*$ and $g = V_{\un{z}} b V_{\un{w}}^*$.
By the covariant condition we have that
\begin{align*}
\psi(fg)
& =
\psi(V_{\un{x} + \un{z} - \un{y} \wedge \un{z}} \al_{\un{z} - \un{y} \wedge \un{z}}(a) \al_{\un{y} - \un{y} \wedge \un{z}}(b) V_{\un{y} + \un{w} - \un{y} \wedge \un{z}}^*)\\
& =
\de_{\un{x} + \un{z} - \un{y} \wedge \un{z}, \un{y} + \un{w} - \un{y} \wedge \un{z}} e^{-\sca{\un{x} + \un{z} - \un{y} \wedge \un{z}, \un{\be}}} \psi(\al_{\un{z}- \un{y} \wedge \un{z}}(a) \al_{\un{y}- \un{y} \wedge \un{z}}(b))\\
& =
\de_{\un{x} + \un{z}, \un{y} + \un{w}} e^{-\sca{\un{x} + \un{z} - \un{y} \wedge \un{z}, \un{\be}}} \psi(\al_{\un{z}- \un{y} \wedge \un{z}}(a) \al_{\un{y}- \un{y} \wedge \un{z}}(b)).
\end{align*}
On the other hand we have that
\begin{align*}
\psi(gf)
& =
\psi(V_{\un{z} + \un{x} - \un{x} \wedge \un{w}} \al_{\un{x} - \un{x} \wedge \un{w}}(b) \al_{\un{w} - \un{x} \wedge \un{w}}(a) V_{\un{y} + \un{w} - \un{x} \wedge \un{w}}^*)\\
& =
\de_{\un{z} + \un{x} - \un{x} \wedge \un{w}, \un{y} + \un{w} - \un{x} \wedge \un{w}} \cdot
e^{- \sca{\un{z} + \un{x} - \un{x} \wedge \un{w}, \un{\be}}}
\cdot \psi(\al_{\un{x} - \un{x} \wedge \un{w}}(b) \al_{\un{w} - \un{x} \wedge \un{w}}(a)) \\
& =
\de_{\un{z} + \un{x}, \un{y} + \un{w}} \cdot
e^{-\sca{\un{z} + \un{x} - \un{x} \wedge \un{w}, \un{\be}}} \cdot \psi(\al_{\un{w} - \un{x} \wedge \un{w}}(a) \al_{\un{x} - \un{x} \wedge \un{w}}(b)),
\end{align*}
where we use that $\psi$ is tracial on $A$.
Since all $\be_k \neq 0$, we have to show that if $\un{z} + \un{x} = \un{y} + \un{w}$ then
\[
\un{z} -  \un{y} \wedge \un{z} = \un{w} - \un{x} \wedge \un{w}
\qand
\un{y} - \un{y} \wedge \un{z} = \un{x} - \un{x} \wedge \un{w},
\]
and that
\[
\un{z} + \un{x} - \un{x} \wedge \un{w}
=
(\un{z} + \un{x} - \un{x} \wedge \un{w}) + (\un{x} - \un{y}).
\]
It suffices to show that
\[
z_k + \min\{x_k, w_k\} = w_k + \min\{y_k,z_k\},
\]
and the symmetrical
\[
x_k + \min\{y_k,z_k\} = y_k + \min\{x_k,w_k\},
\]
for all $k = 1, \dots, n$.
These are immediate since the equation $z_k + x_k = y_k + w_k$ implies that $\min\{y_k,z_k\}= y_k$ if and only if $\min\{x_k,w_k\}=x_k$.
\end{proof}

The next step is to establish the existence of the KMS states when $\la_k\be > 0$ for all $k=1, \dots, n$.

\begin{proposition}\label{P: constr KMS}
Let $\al \colon \bZ_+^n \to \End(A)$ be a unital C*-dynamical system and $\si \colon \bR \to \Aut(\N\T(A,\al))$ be the action related to $\un{\la} \in \bR^n$.
Fix $\be \in \bR$ such that $\la_k \be >0$ for all $k=1, \dots, n$.
Then for every tracial state $\tau$ of $A$ there exists a $(\si, \be)$-KMS state $\psi_{\tau}$ of $\N\T(A,\al)$ such that
\begin{align*}
\psi_{\tau}(V_{\un{x}} a V_{\un{y}}^*)
=
\de_{\un{x}, \un{y}} \cdot
e^{-\sca{\un{x},\un{\la}}\be} \cdot \prod_{i=1}^n (1 - e^{-\la_i\be}) \cdot
\sum_{\un{w} \in \bZ_+^n} e^{-\sca{\un{w},\un{\la}}\be} \tau\al_{\un{w}}(a),
\end{align*}
for all $a\in A$ and $\un{x}, \un{y} \in \bZ_+^n$.
\end{proposition}
\begin{proof}
We write $\un{\be} = \un{\la} \be = (\la_1\be, \dots, \la_n \be)$ so that $\be_k = \la_k \be$ for all $k=1, \dots, n$.
Suppose first that $\psi_{\tau}$ as defined above is a state on $\N\T(A,\al)$.
We will show that it is a KMS state.
For $a,b \in A$ we obtain that
\begin{align*}
\psi_{\tau}(ab)
& =
\prod_{i=1}^n (1 - e^{-\be_i}) \cdot
\sum_{\un{w} \in \bZ_+^n} e^{-\sca{\un{w},\un{\be}}} \tau\al_{\un{w}}(ab) \\
& =
\prod_{i=1}^n (1 - e^{-\be_i}) \cdot
\sum_{\un{w} \in \bZ_+^n} e^{-\sca{\un{w},\un{\be}}} \tau\al_{\un{w}}(ba)
 =
\psi_{\tau}(ba),
\end{align*}
where we have used that $\tau$ is a tracial state on $A$.
Furthermore for $\un{x} \in \bZ_+^n$ and $a \in A$ we readily verify that
\[
e^{-\sca{\un{x},\un{\be}}} \psi(a)= \psi_{\tau}(V_{\un{x}} a V_{\un{x}}^*).
\]
Hence $\psi_{\tau}$ satisfies the conditions of Proposition \ref{P: cond KMS} and it is a $(\si,\be)$-KMS state.

The tricky part is to show that $\psi_{\tau}$ is indeed a state.
To this end consider the Fock representation $(\wt\pi, V_\tau)$ on $H_\tau \otimes \ell^2(\bZ_+^n)$ associated with the GNS representation $(H_\tau, \pi_\tau, \xi_\tau)$ of a state $\tau$ of $A$.
Recall the notation
\[
F_m = \{\un{w} \in \bZ_+^n \mid \un{w} \leq m \cdot \un{1}=(m,\dots,m)\}.
\]
For every $\un{w} \in \bZ_+^n$ let
\[
\psi_{\un{w}}(f) := \sca{(V_\tau \times \wt{\pi_\tau})(f)(\xi_\tau \otimes e_{\un{w}}), \xi_\tau \otimes e_{\un{w}}}
\]
for $f \in \N\T(A,\al)$, and define
\[
\psi_{\tau}(f) := \prod_{i=1}^n(1 - e^{-\be_i}) \cdot \sum_{\un{w} \in \bZ_+^n} e^{- \sca{\un{w}, \un{\be}}} \psi_{\un{w}}(f).
\]
To see that $\psi_{\tau}$ is well defined first note that
\begin{align*}
\sum_{\un{w} \in F_m} e^{- \sca{\un{w},\un{\be}}}
 =
\sum_{w_1=0}^m e^{-w_1\be_1} \cdots \sum_{w_n=0}^m e^{-w_n\be_n}
 =
\prod_{i=1}^n \frac{1 - (e^{-\be_i})^{m+1}}{1 - e^{-\be_i}}
\end{align*}
after a proper re-indexing.
Taking the limit over $m$ yields
\[
\sum_{\un{w} \in \bZ_+^n} e^{- \sca{\un{w},\un{\be}}} = \prod_{i=1}^n(1 - e^{\be_i})^{-1}.
\]
Let a positive element $f \in \N\T(A,\al)$.
Then the sequence of positives $\Big(\sum_{\un{w} \in F_m} e^{-\sca{\un{w},\un{\be}}} \psi_{\un{w}}(f)\Big)_m$ is increasing and bounded above by the sequence $\Big(\sum_{\un{w} \in F_m} e^{-\sca{\un{w},\un{\be}}} \nor{f} \cdot 1 \Big)_m$ which converges to $\nor{f} \cdot \prod_{i=1}^n(1 - e^{\be_i})^{-1} \cdot 1$.
Furthermore a direct computation shows that $\psi_{\un{w}}(1) = 1$ for the unit $1 \in A$ of $\N\T(A,\al)$.
Therefore $\psi_{\tau}(1) = 1$, hence $\psi$ is a state.

It remains to show that $\psi_{\tau}$ is of the form of the statement.
Suppose that $\un{x} \neq \un{y}$.
Then $\psi_{\un{w}}(V_{\un{x}} a V_{\un{y}}^*) = 0$ for all $\un{w} \in \bZ_+^n$, hence $\psi_{\tau}(V_{\un{x}} a V_{\un{y}}^*) = 0$.
Moreover, since $V_{\tau, \un{x}} ^* \xi_\tau \otimes e_{\un{w}} = 0$ when $\un{x} \not\leq \un{w}$, we obtain
\begin{align*}
\psi_{\tau}(V_{\un{x}} a V_{\un{x}}^*)
& =
\prod_{i=1}^n(1 - e^{\be_i}) \cdot \sum_{\un{w} \in \bZ_+^n} e^{-\sca{\un{w},\un{\be}}} \sca{V_{\tau, \un{x}} \wt{\pi_\tau}(a) V_{\tau, \un{x}}^* \xi_\tau \otimes e_{\un{w}}, \xi_\tau \otimes e_{\un{w}}} \\
& =
\prod_{i=1}^n(1 - e^{\be_i}) \cdot \sum_{\un{w} \in \bZ_+^n} e^{-\sca{\un{w},\un{\be}}} \sca{\wt{\pi_\tau}(a) V_{\tau, \un{x}}^* \xi_\tau \otimes e_{\un{w}}, V_{\tau, \un{x}} ^* \xi_\tau \otimes e_{\un{w}}} \\
& =
\prod_{i=1}^n(1 - e^{\be_i}) \cdot \sum_{\un{x} \, \leq \, \un{w} \, \in \, \bZ_+^n} e^{-\sca{\un{w},\un{\be}}} \sca{\pi_\tau\al_{\un{w}-\un{x}}(a) \xi_\tau, \xi_\tau} \\
& =
\prod_{i=1}^n(1 - e^{\be_i}) \cdot \sum_{\un{x} \, \leq \, \un{w} \, \in \, \bZ_+^n} e^{-\sca{\un{w},\un{\be}}} \tau\al_{\un{w} - \un{x}}(a) \\
& =
\prod_{i=1}^n(1 - e^{\be_i}) \cdot \sum_{\un{w} \in \bZ_+^n} e^{-\sca{\un{w} + \un{x}, \un{\be}}} \tau\al_{\un{w}}(a),
\end{align*}
which shows that $\psi$ is as in the statement.
\end{proof}

\subsubsection{Parametrization of KMS states when $\la_k\be >0$}

We will show that the correspondence of Proposition \ref{P: constr KMS} is in fact a parametrization.
We will use the following lemma.

\begin{lemma}\label{L: combinatoric}
Let $(k_{\un{w}})$ be a sequence in $\bC$.
Let $F_m = \{ \un{w} \in \bZ_+^n \mid \un{w} \leq m \cdot \un{1}\}$ be the grids in $\bZ_+^n$ for $m \in \bZ_+$ and fix $\un{y} \leq \un{1}$.
Then we have that
\[
\sum_{\un{0} \leq \un{x} \leq \un{y}} (-1)^{|\un{x}|} \sum_{\un{x} \leq \un{w} \in F_m} k_{\un{w}} = k_{\un{0}}.
\]
\end{lemma}
\begin{proof}
We will show this in the case where $\un{y} = \un{1}$.
The general case follows in the same way.
We will give a proof with a flavour of linear algebra.

Let the finite dimensional space generated by the o.n. basis $\{e_{\un{w}} \mid \un{w} \in F_m\}$.
For $\un{x} \in F_m$ we define $p_{\un{x}}$ be the projection on the subspace generated by $\{e_{\un{w}} \mid \un{w} \geq \un{x}\}$.
Then we obtain that
\[
\sum_{\un{0} \leq \un{x} \leq \un{1}} (-1)^{|\un{x}|} p_{\un{x}} = \prod_{i=1}^n (I - p_{\Bi}),
\]
which in turn is the projection on the subspace $\bC \cdot e_{\un{0}}$.
Furthermore let the operator $T$ on $H$ defined by
\[
T u = \sum_{\un{w} \in F_m} \sca{u,e_{\un{w}}} e_{\un{0}}.
\]
Now let $(k_{\un{w}})$ be a sequence in $\bC$ and define the vector $u = \sum_{\un{w} \in F_m} k_{\un{w}} e_{\un{w}}$.
Then we get that
\begin{align*}
\sum_{\un{0} \leq \un{x} \leq \un{1}} (-1)^{|\un{x}|} \sum_{\un{x} \leq \un{w} \in F_m} k_{\un{w}} e_{\un{0}}
& =
T( \sum_{\un{0} \leq \un{x} \leq \un{1}} (-1)^{|\un{x}|} \sum_{\un{x} \leq \un{w} \in F_m} k_{\un{w}} e_{\un{w}}) \\
& =
T( \sum_{\un{0} \leq \un{x} \leq \un{1}} (-1)^{|\un{x}|} p_{\un{x}} u )\\
& =
T( \prod_{i=1}^n (I - p_{\Bi}) u )
 =
T( k_{\un{0}} e_{\un{0}} )
=
k_{\un{0}} e_{\un{0}}.
\end{align*}
Taking the inner product with $e_{\un{0}}$ then completes the proof.
\end{proof}

\begin{theorem}\label{T: KMS TNP}
Let $\al \colon \bZ_+^n \to \End(A)$ be a unital C*-dynamical system, $\si \colon \bR \to \Aut(\N\T(A,\al))$ be the action related to $\un{\la} \in \bR^n$, and $\be \in \bR$ such that $\la_k \be >0$.
Then there is an affine homeomorphism from the simplex of the tracial states on $A$ onto the simplex of the $(\si,\be)$-KMS states on $\N\T(A,\al)$.
\end{theorem}
\begin{proof}
We write $\un{\be} = \un{\la} \be = (\la_1\be, \dots, \la_n \be)$ so that $\be_k = \la_k \be$ for all $k=1, \dots, n$.

Let the mapping $\tau \mapsto \psi_\tau$ where $\psi_\tau$ is as in Proposition \ref{P: constr KMS}, i.e.
\begin{align*}
\psi_\tau(V_{\un{x}} a V_{\un{y}}^*)
=
\de_{\un{x}, \un{y}} \cdot
e^{-\sca{\un{x},\un{\be}}} \cdot \prod_{i=1}^n (1 - e^{-\be_i}) \cdot
\sum_{\un{w} \in \bZ_+^n} e^{-\sca{\un{w},\un{\be}}} \tau\al_{\un{w}}(a),
\end{align*}
for all $a\in A$ and $\un{x}, \un{y} \in \bZ_+^n$.
The fact that this is an affine weak*-continuous mapping follows by the standard arguments of \cite[Proof of Theorem 6.1]{LRRW13}.

First we show that the mapping is onto.
That is, given a $(\si,\be)$-KMS state $\phi$ of $\N\T(A,\al)$ we will construct a tracial state $\tau$ of $A$ such that $\phi = \psi_\tau$.
By Proposition \ref{P: cond KMS} it suffices to show that $\phi(a) = \psi_\tau(a)$ for all $a\in A$.
A key step is to isolate a projection $P \in \N\T(A,\al)$ such that $Pa=aP$.
To this end let
\[
P = \prod_{i=1}^n (I - V_{\Bi}V_{\Bi}^*),
\]
and recall that $a V_{\Bi} V_{\Bi}^* = V_{\Bi} \al_{\Bi}(a) V_{\Bi}^* = V_{\Bi} V_{\Bi}^* a$ for all $a \in A$.
We remind that we use the notation
\[
F_{m} = \{ \un{x} \in \bZ_+^n \mid \un{x} \leq m \cdot \un{1}\} \qand |\un{x}| = \sum_{k=1}^n x_k.
\]
Then a direct computation shows that
\[
P = \sum_{\un{x} \in F_{1}} (-1)^{|\un{x}|}V_{\un{x}} V_{\un{x}}^*,
\]
where $V_{\un{0}} V_{\un{0}}^* \equiv I$.
Hence we get that
\begin{align*}
\phi(P)
& =
\sum_{\un{x} \in F_{1}} (-1)^{|\un{x}|} \phi(V_{\un{x}} V_{\un{x}}^*)
 =
\sum_{\un{x} \in F_{1}} (-1)^{|\un{x}|} e^{-\sca{\un{x},\un{\be}}}
 =
\prod_{i=1}^n (1 - e^{-\be_i}).
\end{align*}
Note here that $\phi(P)$ is constant for all $(\si,\be)$-KMS states $\phi$.
We claim that the function
\[
\phi_P \colon A \to \bC: a \mapsto \frac{\phi(PaP)}{\phi(P)}
\]
is a tracial state on $A$.
Indeed, for $a=1 \in A$ (which is also the unit of $\N\T(A,\al)$) we get that $\phi_P(1) = 1$.
Moreover by using that $\phi$ is a $(\si,\be)$-KMS state and that $\si_{i\be}(a) = a$ for all $a\in A$ and $\si_{i\be}(P) = P$ we get that
\[
\phi(PabP) = \phi(bPPa) = \phi(bPa)= \phi(Pba) = \phi(PbaP).
\]
Thus $\phi_P(ab) = \phi_P(ba)$ for all $a,b \in A$.

Secondly we construct a sequence of projections $p_m \in \N\T(A,\al)$ such that $\lim_m \phi(p_m) =1$.
To this end let
\[
p_m = \sum_{\un{w} \in F_m} V_{\un{w}} P V_{\un{w}}^*.
\]
Note that $P V_{\Bi} = 0$ and consequently $PV_{\un{x}} = 0$ for all $\un{x} \neq \un{0}$.
Thus we get that
\[
P V_{\un{y}}^* V_{\un{x}} P = P V_{\un{x} - \un{y} \wedge \un{x}} V_{\un{y} - \un{y} \wedge \un{x}}^* P = \de_{\un{x}, \un{y}} P
\]
for all $\un{x}, \un{y} \in \bZ_+^n$.
Therefore each $p_m$ is a projection.
We compute
\begin{align*}
\phi(p_m)
& =
\sum_{\un{w} \in F_m} \phi(V_{\un{w}} P V_{\un{w}}^*) \\
& =
\sum_{\un{w} \in F_m}  e^{- \sca{\un{w},\un{\be}}} \phi(P) \\
& =
\phi(P) \sum_{\un{w} \in F_m} e^{-\sca{\un{w},\un{\be}}},
\end{align*}
where $\phi(P) = \prod_{i=1}^n (1 - e^{-\be_i})$.
However after a proper re-indexing we get that
\begin{align*}
\lim_m \sum_{\un{w} \in F_m} e^{-\sca{\un{w},\un{\be}}}
& =
\lim_m \prod_{i=1}^n \frac{1 - (e^{-\be_i})^m}{1- e^{-\be_i}} \\
& =
\prod_{i=1}^n (1 - e^{-\be_i})^{-1}
 =
\phi(P)^{-1},
\end{align*}
hence $\lim_m \phi(p_m) = 1$.
Therefore $\lim_m \phi(p_m f p_m) = \phi(f)$ for all $f \in \N\T(A,\al)$ by \cite[Lemma 7.3]{LRR11}.
In particular for $a \in A$ we get that
\begin{align*}
\phi(a)
& =
\lim_m \phi(p_m a p_m) \\
& =
\lim_m \sum_{\un{w} \in F_m} \sum_{\un{z} \in F_m} \phi(V_{\un{w}} P V_{\un{w}}^* \cdot a \cdot V_{\un{z}} P V_{\un{z}}^*) \\
& =
\lim_m \sum_{\un{w} \in F_m} \sum_{\un{z} \in F_m} e^{-\sca{\un{w},\un{\be}}} \phi(P V_{\un{w}}^* a V_{\un{z}} P V_{\un{z}}^* \cdot V_{\un{w}} P) \\
& =
\lim_m \sum_{\un{w} \in F_m} \sum_{\un{z} \in F_m} e^{-\sca{\un{w},\un{\be}}} \de_{\un{w}, \un{z}} \phi(P V_{\un{w}}^* a V_{\un{z}} P) \\
& =
\lim_m \sum_{\un{w} \in F_m} e^{- \sca{\un{w},\un{\be}}} \phi(P V_{\un{w}}^* a V_{\un{w}} P) \\
& =
\lim_m \sum_{\un{w} \in F_m}  e^{-\sca{\un{w},\un{\be}}} \phi(P \al_{\un{w}}(a) P) \\
& =
\prod_{i=1}^n (1 - e^{-\be_i}) \cdot \sum_{\un{w} \in \bZ_+^n} \phi_P\al_{\un{w}}(a),
\end{align*}
thus $\phi$ coincides with $\psi_{\tau}$ for the trace $\tau = \phi_P$.

For the last step let $\psi_\tau$ be as in Proposition \ref{P: constr KMS}.
It then suffices to show that
\[
(\psi_\tau)_P(a) = \tau(a) \foral a \in A.
\]
Indeed in this case we get that if $\psi_\tau = \psi_\rho$ then $\tau = \rho$ and thus the correspondence $\tau \mapsto \psi_\tau$ is one-to-one.
Recall that $\psi_\tau$ is a KMS state and that $\si_{i\be}(P) = P$, hence
\[
\psi_\tau(PaP) = \psi_\tau(a P P) = \psi_\tau(aP) \foral a \in A.
\]
We have to show that
\[
\psi_\tau(a P) = \frac{\tau(a)}{\psi_\tau(P)} \foral a \in A.
\]
Recall that $P = \sum_{\un{0} \leq \un{x} \leq \un{1}} (-1)^{|\un{x}|} V_{\un{x}} V_{\un{x}}^*$ and compute
\begin{align*}
\psi_\tau(a P)
& =
\psi_\tau\Big( \sum_{\un{0} \leq \un{x} \leq \un{1}} (-1)^{|\un{x}|} a V_{\un{x}} V_{\un{x}}^* \Big)
 =
\sum_{\un{0} \leq \un{x} \leq \un{1}} (-1)^{|\un{x}|} \psi_\tau(V_{\un{x}} \al_{\un{x}}(a) V_{\un{x}}^*)
\end{align*}
For convenience let us write
\[
p = \psi_\tau(P)^{-1} = \prod_{i=1}^n (1 - e^{-\be_i})^{-1}.
\]
Then we get that
\begin{align*}
\psi_\tau(a P)
& =
p \sum_{\un{0} \leq \un{x} \leq \un{1}} (-1)^{|\un{x}|} e^{-\sca{\un{x},\un{\be}}} \sum_{\un{w} \in \bZ_+^n} e^{-\sca{\un{w},\un{\be}}} \tau \al_{\un{w}}(\al_{\un{x}}(a)) \\
& =
p \sum_{\un{0} \leq \un{x} \leq \un{1}} (-1)^{|\un{x}|} \sum_{\un{w} \in \bZ_+^n} e^{-\sca{\un{x} + \un{w},\un{\be}}} \tau \al_{\un{x} + \un{w}}(a) \\
& =
p \sum_{\un{0} \leq \un{x} \leq \un{1}} (-1)^{|\un{x}|} \sum_{\un{x} \leq \un{w} \in \bZ_+^n} e^{-\sca{\un{w},\un{\be}}} \tau \al_{\un{w}}(a) \\
& =
p \, \lim_m \sum_{\un{0} \leq \un{x} \leq \un{1}} (-1)^{|\un{x}|} \sum_{\un{x} \leq \un{w} \in F_m} e^{-\sca{\un{w}, \un{\be}}} \tau \al_{\un{w}} (a).
\end{align*}
By Lemma \ref{L: combinatoric} we then get that
\[
\sum_{\un{0} \leq \un{x} \leq \un{1}} (-1)^{|\un{x}|} \sum_{\un{x} \leq \un{w} \in F_m} e^{-\sca{\un{w}, \un{\be}}} \tau \al_{\un{w}} (a) = e^{-\sca{\un{0}, \un{\be}}} \tau \al_{\un{0}} (a) = \tau(a)
\]
and therefore $\psi_\tau(aP) = p \tau(a)$ as required.
\end{proof}

\subsubsection{Ground states}

As in \cite{LacRae10, LRRW13} we give a characterization of the ground states.
This will imply their existence and association with the states on the C*-algebra $A$.

\begin{proposition}\label{P: ground TNP}
Let $\al \colon \bZ_+^n \to \End(A)$ be a unital C*-dynamical system.
Then a state $\psi$ of $\N\T(A,\al)$ is a ground state if and only if
\[
\psi(V_{\un{x}} a V_{\un{y}}^*) =
\begin{cases}
\psi(a) & \text{ for } \un{x}=\un{0}=\un{y},\\
0 & \text{ otherwise},
\end{cases}
\]
for a state $\tau$ of $A$.

Consequently, there is an affine homeomorphism from the state space $S(A)$ onto the ground states on $\N\T(A,\al)$.
\end{proposition}
\begin{proof}
Let $\si \colon \bR \to \Aut(A)$ be the action related to $\un{\la} \in \bR^n$ and recall that  $\si_z(V_{\un{x}} a V_{\un{y}}^*) = e^{i \sca{\un{x} - \un{y}, \un{\la}} z} V_{\un{x}} a V_{\un{y}}^*$ for all $z \in \bC$.

First fix a ground state $\psi$.
Suppose that $\un{y} \neq \un{0}$ and note that the map
\[
r + i t \mapsto \psi(V_{\un{x}} \si_{r + it}(a V_{\un{y}}^*)) = e^{-i\sca{\un{y},\un{\la}}r} e^{\sca{\un{y},\un{\la}}t} \psi(V_{\un{x}} a V_{\un{y}}^*)
\]
must be bounded when $t > 0$, for all $a \in A$.
Therefore we get that $\psi(V_{\un{x}} a V_{\un{y}}^*) =0$.
When $\un{y} = \un{0}$ but $\un{x} \neq \un{0}$, we have that the function
\[
r + it \mapsto \psi(a^* \si_{r + i t}(V_{\un{x}}^*)) = e^{-i\sca{\un{x},\un{\la}}r} e^{\sca{\un{x},\un{\la}}t} \psi(a^* V_{\un{x}}^*)
\]
must be bounded for $t > 0$; thus $\psi(a^* V_{\un{x}}^*) = 0$.
Taking adjoints yields $\psi(V_{\un{x}} a) =0$.

Conversely let $\psi$ be a state on $\N\T(A,\al)$ that satisfies the condition of the statement.
Then for $f = V_{\un{x}} a V_{\un{y}}^*$ and $g = V_{\un{z}} b V_{\un{w}}^*$ we compute
\begin{align*}
|\psi(f \si_{r + it}(g))|^2
& =
|e^{i\sca{\un{z} - \un{w}, \un{\la}}(r + it)} \psi(fg)|^2 \\
& =
e^{-\sca{\un{z} - \un{w}, \un{\la}}t} |\psi(fg)|^2 \\
& \leq
e^{-\sca{\un{z} - \un{w}, \un{\la}}t} \psi(f^*f) \psi(g^*g) \\
& \leq
e^{-\sca{\un{z} - \un{w}, \un{\la}}t} \psi(V_{\un{y}} a^* a V_{\un{y}}^*) \psi(V_{\un{w}} b^*bV_{\un{w}}^*).
\end{align*}
When $\un{y} \neq \un{0}$ or $\un{w} \neq \un{0}$ then the above expression is $0$.
When $\un{y} = \un{w} = \un{0}$ then
\begin{align*}
|\psi(f \si_{r + it}(g))|
& =
e^{-\sca{\un{z},\un{\la}}t} |\psi(V_{\un{x}} a V_{\un{z}} b)|
 =
e^{-\sca{\un{z},\un{\la}}t} |\psi(V_{\un{x} + \un{z}} \al_{\un{z}}(a) b)|,
\end{align*}
which is zero when $\un{x} \neq \un{0}$ or $\un{z} \neq \un{0}$.
Finally when $\un{x} = \un{y} = \un{z} = \un{w} = \un{0}$ then $\psi(f \si_{r + it}(g)) = \psi(ab)$, which is bounded on $\{z \in \bC \mid \text{Im}(z)>0\}$ for all $a,b \in A$.

Because of this characterization, every state on $A$ gives rise to a ground state on $\N\T(A,\al)$.
Indeed, let $\tau$ be a state on $A$ and let $(H_\tau, \pi_\tau, \xi_\tau)$ be the associated GNS representation.
Then for the Fock representation $(\wt{\pi_\tau}, V_{\tau})$ we define the state
\[
\psi_\tau(f) = \sca{f \xi_\tau \otimes e_{\un{0}}, \xi_\tau \otimes e_{\un{0}}}, \foral f \in \N\T(A,\al).
\]
It is readily verified that $\psi_\tau(a) = \tau(a)$ for all $a\in A$.
For $f = V_{\un{x}} a V_{\un{y}}^*$ we compute
\begin{align*}
\psi_\tau(V_{\un{x}} a V_{\un{y}}^*)
& =
\sca{V_{\tau, \un{x}} \wt{\pi_\tau}(a) V_{\tau, \un{y}}^* \xi_\tau \otimes e_{\un{0}}, \xi_\tau \otimes e_{\un{0}}} \\
& =
\sca{\wt{\pi_\tau}(a) V_{\tau, \un{y}}^* \xi_\tau \otimes e_{\un{0}}, V_{\tau, \un{x}}  \xi_\tau \otimes e_{\un{0}}} \\
& =
\de_{\un{x}, \un{0}} \de_{\un{y}, \un{0}} \sca{\pi_\tau(a) \xi_\tau, \xi_\tau} \\
& =
\begin{cases}
\tau(a) & \text{ for } \un{x}=\un{0}=\un{y},\\
0 & \text{ otherwise},
\end{cases} \\
& =
\begin{cases}
\psi_\tau(a) & \text{ for } \un{x}=\un{0}=\un{y},\\
0 & \text{ otherwise}.
\end{cases}
\end{align*}
Finally note that $\psi = \psi_{\tau}$ for $\tau = \psi|_A$, since the restriction of a ground state to $A$ defines a state on $A$.
\end{proof}

\subsubsection{KMS$_\infty$ states}

We continue with the characterization of the KMS$_\infty$ states on $\N\T(A,\al)$.

\begin{proposition}\label{P: infty TNP}
Let $\al \colon \bZ_+^n \to \End(A)$ be a unital C*-dynamical system.
Then a state $\psi$ is a KMS$_\infty$ state on $\N\T(A,\al)$ if and only if
\[
\psi(V_{\un{x}} a V_{\un{y}}^*)
=
\begin{cases}
\tau(a) & \text{ for } \un{x}=\un{0}=\un{y},\\
0 & \text{ otherwise},
\end{cases}
\]
for a tracial state $\tau$ of $A$.

Consequently, there is an affine homeomorphism from the tracial state space $T(A)$ onto the KMS${}_\infty$ states on $\N\T(A,\al)$.
\end{proposition}
\begin{proof}
Fix $\psi$ be a KMS$_\infty$ state on $\N\T(A,\al)$.
Let $\psi_{\be}$ be $(\si,\be)$-KMS states on $\N\T(A,\al)$ that converge in the w*-topology to $\psi$.
By Proposition \ref{P: cond KMS} we obtain that $\psi|_A$ is a tracial state on $A$ and that
\[
\psi_{\be}(V_{\un{x}} a  V_{\un{y}}^*) = \de_{\un{x}, \un{y}} e^{-\sca{\un{x},\un{\la}}\be} \psi(a),
\]
which tends to zero when $\be \longrightarrow \infty$, if $\un{x} \neq \un{0}$ or $\un{y} \neq \un{0}$.

Conversely, let $\psi_\tau$ be as in the statement with respect to a tracial state $\tau$ of $A$.
Let $\psi_{\tau,\be}$ be as defined in Proposition \ref{P: constr KMS}, i.e.
\begin{align*}
\psi_{\tau, \be}(V_{\un{x}} a V_{\un{y}}^*)
=
\de_{\un{x}, \un{y}} \cdot
e^{-\sca{\un{x},\un{\la}}\be} \prod_{i=1}^n (1 - e^{-\la_i\be}) \cdot
\sum_{\un{w} \in \bZ_+^n} e^{-\sca{\un{w},\un{\la}}\be} \tau\al_{\un{w}}(a).
\end{align*}
By the w*-compactness we may choose a sequence of such states that converges to a state, say $\psi$.
By definition $\psi$ is then a KMS$_\infty$ state, and we aim to show that $\psi_\tau = \psi$.
When $\un{x}, \un{y} \neq \un{0}$ then we get that $\lim_{\be \to \infty} \psi_{\tau,\be}(V_{\un{x}} a V_{\un{y}}^*) = 0$, as in the preceding paragraph.
When $\un{x} = \un{y} =0$ then
\begin{align*}
\psi_{\tau,\be}(a)
& =
\prod_{i=1}^n (1 - e^{-\la_i\be}) \cdot
\sum_{\un{w} \in \bZ_+^n} e^{-\sca{\un{w},\un{\la}}\be} \tau\al_{\un{w}}(a) \\
& =
\prod_{i=1}^n (1 - e^{-\la_i\be}) \cdot
\Big( \tau(a) + \sum_{\un{w} > \un{0}} e^{-\sca{\un{w},\un{\la}}\be} \tau\al_{\un{w}}(a) \Big).
\end{align*}
However
\begin{align*}
| \sum_{\un{w} > \un{0}} e^{-\sca{\un{w},\un{\la}}\be} \tau\al_{\un{w}}(a) |
& \leq
\nor{a} \cdot \sum_{\un{w} > \un{0}} e^{-\sca{\un{w},\un{\la}}\be} \\
& =
\nor{a}(-1 + \prod_{i=1}^n(1 - e^{-\la_i\be})^{-1}).
\end{align*}
Taking $\be \longrightarrow \infty$ in the last expression yields that $\sum_{\un{w} > \un{0}} e^{-\sca{\un{w},\un{\la}}\be} \tau\al_{\un{w}}(a)$ tends to zero.
Trivially $\lim_{\be \to \infty} \prod_{i=1}^n (1 - e^{-\la_i\be}) = 1$, which implies that $\lim_{\be \rightarrow \infty} \psi_{\tau,\be}(a) = \tau(a)$, hence $\psi = \psi_\tau$.
The proof is completed as in Proposition \ref{P: ground TNP}.
\end{proof}

\subsubsection{KMS states at $\be=0$ (tracial states)}

A tracial state on $\N\O(A,\al)$ defines automatically a tracial state on $\N\T(A,\al)$.
By Proposition \ref{P: tracial} (that will follow) this works also in the converse direction.

\begin{proposition}
Let $\al \colon \bZ_+^n \to \End(A)$ be a unital C*-dynamical system and let $\wt{\be} \colon \bZ^n \to \Aut(\wt{B})$ be its automorphic dilation as defined in Subsection \ref{Ss: CNP}.
For any tracial state $\tau$ of $\wt{B}$ there exists a tracial state $\psi$ of $\N\T(A,\al)$ such that
\begin{align*}
\psi (V_{\un{x}} a V_{\un{y}}^*) = \de_{\un{x}, \un{y}} \tau\wt{\be}_{-\un{x}}(a) \foral a\in A, \un{x}, \un{y} \in \bZ_+^n.
\end{align*}
\end{proposition}
\begin{proof}
It suffices to find a tracial state $\phi$ of $\N\O(A,\al)$ so that $\phi (U_{\un{x}} a U_{\un{y}}^*) = \de_{\un{x}, \un{y}} \tau(\wt{\be}_{-\un{x}}(a))$.
Then we may set $\psi = \phi q$ for the canonical $*$-epimorphism $q \colon \N\T(A,\al) \to \N\O(A,\al)$.
Recall that $\N\O(A,\al)$ is a corner of the crossed product $\wt{B} \rtimes_{\wt{\be}} \bZ^n$ by the element $p=1 \in A \subseteq \wt{B}$.
Consequently the tracial states on $\wt{B} \rtimes_{\wt{\be}} \bZ^n$ define tracial states on $\N\O(A,\al)$ by restriction.
If $\tau$ is a tracial state on $\wt{B}$ then $\psi : = \tau E \colon \wt{B} \rtimes_{\wt{\be}} \bZ^n \to \bC$ is a tracial state, where $E$ is the conditional expectation on the crossed product.
It is then readily verified that $\psi(U_{\un{x}} a U_{\un{y}}^*) = \de_{\un{x}, \un{y}} \tau\wt{\be}_{-\un{x}}(a)$ for all $a\in A$ and $\un{x}, \un{y} \in \bZ_+^n$.
\end{proof}

\subsubsection{Applications}

Our analysis can be used to treat the particular cases when $\la_k=1$ for all $k=1, \dots, n$, or when $\la_k=0$ for some $k=1, \dots, n$.

\begin{example}
Let the action $\si \colon \bR \to \Aut(\N\T(A,\al))$ given by $\si_t = \ga_{(\exp(it), \dots, \exp(it))}$.
Then the KMS condition translates into
\begin{align*}
\psi(V_{\un{x}} aV_{\un{y}}^* \cdot V_{\un{z}} b V_{\un{w}}^*)
& =
e^{-(|\un{x}|- |\un{y}|)\be}\psi(V_{\un{z}} b V_{\un{w}}^* \cdot V_{\un{x}} a V_{\un{y}}^*),
\end{align*}
for all $a,b \in A$ and $\un{x}, \un{y}, \un{z}, \un{w} \in \bZ_+^n$.
The previous analysis then gives the appropriate characterization of the $(\si,\be)$-KMS states by setting $\la_k=1$ for all $k=1, \dots, n$:

\begin{inparaenum}[\upshape(i)]
\item for $\be<0$ there are no $(\si,\be)$-KMS states;

\item for $\be > 0$, a state $\psi$ is $(\si,\be)$-KMS state if and only if $\psi(ab)=\psi(ba)$, and $\psi(V_{\un{x}} a V_{\un{y}}^*) = \de_{\un{x}, \un{y}} e^{-|\un{x}|\be} \psi(a)$, for all $a,b \in A$ and $n,m \in \bZ_+$;

\item for every tracial state $\tau$ of $A$ and $\be>0$ there is a $(\si,\be)$-KMS state $\psi_\tau$ of $\N\T(A,\al)$ such that
    \[
    \psi_\tau(V_{\un{x}} a V_{\un{y}}^*)
    =
    \de_{\un{x}, \un{y}} \cdot
    e^{-|\un{x}|\be} \cdot (1 - e^{-\be})^n \sum_{\un{w} \in \bZ_+^n} e^{-|\un{w}|\be} \tau\al_{\un{w}}(a),
    \]
    for all $a\in A$ and $\un{x}, \un{y} \in \bZ_+^n$. This representation is a parametrization of the $(\si,\be)$-KMS states.
\end{inparaenum}
\end{example}

Now let us examine the case where some $\la_k$ are zeroes.
Without loss of generality we may assume that $\la_{d+1} = \dots = \la_n =0$ and $\la_1, \dots, \la_d \neq 0$.
Then $\si \colon \bR \to \Aut(\N\T(A,\al))$ is given by
\[
\si_{t} = \ga_{(\exp(i\la_1t), \dots, \exp(i\la_dt), 1, \dots, 1)}.
\]
Hence we obtain
\begin{align*}
\si_t(V_{\un{x}} a V^*_{\un{y}}) = e^{i \sum_{k=1}^d(x_{k} - y_{k})t} V_{\un{x}} a V_{\un{y}}^*,
\end{align*}
and the KMS condition is translated into
\begin{align*}
\psi(V_{\un{x}} a V_{\un{y}}^* \cdot V_{\un{z}} b V_{\un{w}}^*)
& =
e^{-\sum_{k=1}^d(x_k-y_k)\la_k\be} \psi(V_{\un{z}} b V_{\un{w}}^* \cdot V_{\un{x}} a V_{\un{y}}^*),
\end{align*}
for all $a,b \in A$ and $\un{x}, \un{y}, \un{z}, \un{w} \in \bZ_+^n$.
We aim to show that the $(\si,\be)$-KMS states on $\N\T(A,\al)$ are determined by the $(\si',\be)$-KMS states on $\N\T(Z,\zeta)$ for a suitably chosen C*-dynamical system $\zeta \colon Z \to \End(Z)$ and an appropriate action $\si' \colon \bR \to \Aut(\N\T(Z,\zeta))$.

To this end recall that $a V_{\Bi} = V_{\Bi} \al_{\Bi}(a)$ for all $a\in A$.
Then $a V_{\Bi} V_{\Bi}^* = V_{\Bi} V_{\Bi}^* a$ and since $V_{\Bi}$ is an isometry we obtain that
\[
\al_{\Bi}(a) = V_{\Bi}^* a V_{\Bi}, \foral a \in A.
\]
We define the C*-algebra
\[
Z = \ol{\spn}\{V_{\un{x}} a V_{\un{y}}^* \mid a\in A, \supp \un{x}, \supp \un{y} \subseteq \{\bo{d+1}, \dots, \bo{n}\} \}.
\]
Observe that $Z$ is a C*-subalgebra of $\N\T(A,\al)$.
For $\Bi = \bo{1}, \dots, \bo{d}$ we can extend the $*$-endomorphism $\al_\Bi$ to a $*$-endomorphism $\zeta_{\Bi}$ of $Z$ defined by
\[
\zeta_{\Bi} (f) = V_{\Bi}^* f V_{\Bi}, \foral f \in Z.
\]
To see that every $\zeta_{\Bi}$ is an algebraic endomomorphism we remark that $V_{\bo{i}} V_{\bo{i}}^* \in Z'$ and that
\begin{align*}
\zeta_{\Bi} (V_{\un{x}} a V_{\un{y}}^*)
 =
V_{\Bi}^* V_{\un{x}} a V_{\un{y}}^* V_{\Bi}
 =
V_{\un{x}} \al_{\Bi}(a) V_{\un{y}}^* V_{\Bi}^* V_{\Bi}
 =
V_{\un{x}} \al_{\Bi}(a) V_{\un{y}}^*,
\end{align*}
where we have used that the $V_{\Bi}$ are doubly commuting isometries and that $\un{x}, \un{y} \in \Bi^\perp$.
This is a concrete system and $\N\T(Z,\zeta)$ has a representation in $\N\T(A,\al)$ given by the covariant pair $(\id_Z,V)$.
We will show that the induced $*$-representation is faithful.

\begin{proposition}
With the aforementioned notation, let $Z \subseteq \N\T(A,\al)$, and let the covariant pair $(\id_Z,V)$ of $\zeta \colon \bZ_+^n \to \End(Z)$.
Then the induced $*$-representation on $\N\T(Z,\zeta)$ is faithful.
\end{proposition}
\begin{proof}
Note that the representation $(\id_Z,V)$ admits a gauge action given by
\[
\ga' \colon \bT^d \to \Aut(\N\T(Z,\zeta)),
\]
such that $\ga_{\un{z}}' = \ad_{u_{\un{z}}}$ where $u_{\un{z}}(\xi \otimes e_{\un{w}}) = z_1^{w_1} \dots z_d^{w_d} \xi \otimes e_{\un{w}}$.
Therefore by the gauge invariant uniqueness theorem it suffices to show that $Z \bigcap \B_{(\un{0},\infty]} = (0)$, where
\[
\B_{(\un{0},\infty]} = \ol{\spn}\{V_{\un{x}} f V_{\un{x}}^* \mid f \in Z, \supp \un{x} \subseteq \{\bo{1}, \dots, \bo{d}\} \}.
\]
For a typical monomial $V_{\un{x}} a V_{\un{y}}^* \in Z$ we note that $V_{\un{x}} a V_{\un{y}}^* (\xi \otimes e_{\un{w}}) = 0$, when $\supp \un{w} \subseteq \{\bo{1}, \dots, \bo{d}\}$.
On the other hand $\B_{(\un{0},\infty]}$ is the inductive limit of the C*-subalgebras
\[
\B_{(\un{0}, m \cdot \bo{d}]} = \spn\{V_{\un{x}} f V_{\un{x}}^* \mid f \in Z, \un{0}< \un{x} \leq m \cdot \bo{d} \}.
\]
For $X \in \B_{(\un{0},m \cdot \bo{d}]}$ it is immediate to deduce that if $X(\xi \otimes e_{\un{w}}) =0$ for all $\un{0}< \un{x} \leq m \cdot \bo{d}$ then $X=0$.
Inductively for $X \in \B_{(\un{0},\infty]}$ we obtain that if $X(\xi \otimes e_{\un{w}}) = 0$ then $X =0$.
This shows that $Z \bigcap \B_{(\un{0},\infty]} = (0)$.
\end{proof}

An immediate corollary is that $\N\T(A,\al) = \N\T(Z,\zeta)$.
Then the action $\si$ induces an action on $\N\T(Z,\zeta)$.
The gain is that $\si_t|_Z = \id_Z$ and we fall into the previous analysis.
In particular note that there is a bijection between the actions $\si$ of $\N\T(A,\al)$ related to $\un{\la}$ with $\la_{d+1} = \dots = \la_n =0$ and the actions $\si'$ of $\N\T(Z,\zeta)$ defined by
\[
\si'_t = \ga'_{(\exp(i\la_1t), \dots, \exp(i\la_dt))},
\]
where $\{\ga'_{\un{z}}\}_{\un{z} \in \bT^d}$ here is the gauge action of $\N\T(Z,\zeta)$.

\begin{proposition}\label{P: d-states TNP}
With the aforementioned notation, a state $\psi$ is a $(\si, \be)$-KMS state on $\N\T(A,\al)$ if and only if $\psi$ is a $(\si',\be)$-KMS state on $\N\T(Z,\zeta)$.
\end{proposition}
\begin{proof}
We will use that every element $\un{x} \in \bZ_+^n$ is decomposed as $\un{x} = \un{x}_d + \un{x}_{n-d}$ where
\[
\un{x}_d = (x_1, \dots, x_d, 0, \dots, 0) \qand \un{x}_{n-d} = (0, \dots, 0, x_{d+1}, \dots, x_n).
\]
Therefore a typical element $V_{\un{x}} a V_{\un{y}}^*$ can be written as $V_{\un{x}_d} f V_{\un{y}_d}^*$, with $f = V_{\un{x}_{n-d}} a V_{\un{y}_{n-d}}^* \in Z$.

First suppose that $\psi$ is a $(\si,\be)$-KMS state on $\N\T(A,\al)$.
We have to establish the equality
\begin{align*}
\psi(V_{\un{x}_d} f V_{\un{y}_d}^* \cdot V_{\un{z}_d} g V_{\un{w}_d}^*)
& =
\psi(V_{\un{z}_d} g V_{\un{w}_d}^* \cdot \si_{i\be}'(V_{\un{x}_d} f V_{\un{y}_d}^*)) \\
& =
e^{-\sum_{k=1}^d(x_k-y_k)\la_k\be} \psi(V_{\un{z}_d} g V_{\un{w}_d}^* \cdot V_{\un{x}_d} f V_{\un{y}_d}^*),
\end{align*}
for $f, g \in \N\T(Z,\zeta)$ and $\un{x}, \un{y}, \un{z}, \un{w} \in \bZ_+^n$.
It suffices to check it for $f = V_{\un{x}_{n-d}} a V_{\un{y}_{n-d}}^*$ and $g = V_{\un{z}_{n-d}} b V_{\un{w}_{n-d}}^*$, since such elements span a dense subset of $Z$.
This follows directly by the computation
\begin{align*}
\psi(V_{\un{x}_d} f V_{\un{y}_d}^* \cdot V_{\un{z}_d} g V_{\un{w}_d}^*)
& =
\psi(V_{\un{x}} a V_{\un{y}}^* \cdot V_{\un{z}} b V_{\un{w}}^*) \\
& =
\psi(V_{\un{z}} b V_{\un{w}}^* \cdot \si_{i\be}(V_{\un{x}} a V_{\un{y}}^*)) \\
& =
e^{-\sum_{k=1}^d(x_k-y_k)\la_k\be} \psi(V_{\un{z}} b V_{\un{w}}^* \cdot V_{\un{x}} a V_{\un{y}}^*) \\
& =
e^{-\sum_{k=1}^d(x_k-y_k)\la_k\be} \psi(V_{\un{z}_d} g V_{\un{w}_d}^* \cdot V_{\un{x}_d} f V_{\un{y}_d}^*).
\end{align*}

Conversely suppose that $\psi$ is a $(\si',\be)$-KMS state on $\N\T(Z,\zeta)$.
We have to establish the equality
\begin{align*}
\psi(V_{\un{x}} a V_{\un{y}}^* \cdot V_{\un{z}} b V_{\un{w}}^*)
& =
e^{-\sum_{k=1}^d(x_k-y_k)\la_k\be} \psi(V_{\un{z}} b V_{\un{w}}^* \cdot V_{\un{x}} a V_{\un{y}}^*),
\end{align*}
for $a,b \in A$ and $\un{x}, \un{y}, \un{z}, \un{w} \in \bZ_+^n$.
This follows directly by the computation
\begin{align*}
\psi(V_{\un{x}} a V_{\un{y}}^* \cdot V_{\un{z}} b V_{\un{w}}^*)
& =
\psi(V_{\un{x}_d} f V_{\un{y}_d}^* \cdot V_{\un{z}_d} g V_{\un{w}_d}^*) \\
& =
\psi(V_{\un{z}_d} g V_{\un{w}_d}^* \cdot \si_{i \un{\be}}'(V_{\un{x}_d} f V_{\un{y}_d}^*)) \\
& =
e^{-\sum_{k=1}^d(x_k-y_k)\la_k\be} \psi(V_{\un{z}_d} g V_{\un{w}_d}^* \cdot V_{\un{x}_d} f V_{\un{y}_d}^*) \\
& =
e^{-\sum_{k=1}^d(x_k-y_k)\la_k\be} \psi(V_{\un{z}} b V_{\un{w}}^* \cdot V_{\un{x}} a V_{\un{y}}^*),
\end{align*}
for the elements $f = V_{\un{x}_{n-d}} a V_{\un{y}_{n-d}}^* \in Z$ and $g = V_{\un{z}_{n-d}} b V_{\un{w}_{n-d}}^* \in Z$.
\end{proof}

\subsection{The Cuntz-Nica-Pimsner algebra}

Recall that $\N\O(A,\al)$ is the quotient of $\N\T(A,\al)$ by the ideal generated by the elements
\[
a \cdot \prod_{\Bi \in \supp \un x} (I - V_\Bi V_\Bi^*), \foral a \in I_{\un x},
\]
and for all $\un{x} > \un{0}$, where
\[
I_{\un x} = \bigcap_{\un y \in \un x^\perp} \al_{\un y}^{-1}
\Big(\big( \bigcap_{\Bi\in \supp \un x } \ker\al_\Bi \big)^\perp\Big).
\]
We denote by $q \colon \N\T(A,\al) \to \N\O(A,\al)$ the quotient map.
Since $A$ embeds isometrically inside $\N\O(A,\al)$ we write $a \equiv q(a)$ for all $a\in A$.
Moreover we write $U_{\un{x}} = q(V_{\un{x}})$ for all $\un{x} \in \bZ_+^n$.
The action $\si$ of $\N\T(A,\al)$ passes naturally to an action of $\N\O(A,\al)$ which we denote by the same symbol.
It is readily verified that the $(\si,\be)$-KMS states on $\N\O(A,\al)$ define $(\si,\be)$-KMS states on $\N\T(A,\al)$.
The converse is true when the state vanishes on $\ker q$.

\subsubsection{KMS states at $\be=0$}

The algebra $\N\O(A,\al)$ is a C*-subalgebra of $\wt{B^{(1)}} \rtimes_{\wt{\be^{(1)}}} \bZ^n$. Therefore $\N\O(A,\al)$ admits restrictions of tracial states on this crossed product.
On the other hand it shares the ``same'' tracial states with $\N\T(A,\al)$.

\begin{proposition}\label{P: tracial}
Let $\al \colon \bZ_+^n \to \End(A)$ be a unital C*-dynamical system.
Then a state is tracial on $\N\T(A,\al)$ if and only if it factors through a tracial state on $\N\O(A,\al)$.
\end{proposition}
\begin{proof}
It suffices to show that if $\psi$ is a tracial state on $\N\T(A,\al)$ then it vanishes on $\ker q$ which is generated by the elements
\[
a \cdot \prod_{\Bi \in \supp \un x} (I - V_\Bi V_\Bi^*), \foral a \in I_{\un x} \text{ and } \un{x} > \un{0}.
\]
To this end fix $\un{x} > 0$ such that $x_k \neq 0$.
Since $\psi$ is tracial we obtain that
\[
1 = \psi(I) = \psi(V_{\bo{k}}^* V_{\bo{k}}) = \psi(V_{\bo{k}} V_{\bo{k}}^*).
\]
By the Cauchy-Schwartz inequality we then obtain that
\begin{align*}
|\psi( a \cdot \prod_{\Bi \in \supp \un x} (I - V_\Bi V_\Bi^*))|^2
& \leq
M \cdot \psi\left((I - V_{\bo{k}} V_{\bo{k}}^*) (I - V_{\bo{k}} V_{\bo{k}}^*)^* \right) \\
& =
M \cdot \psi(I - V_{\bo{k}} V_{\bo{k}}^*) = 0,
\end{align*}
for all $a\in A$, since $I - V_{\bo{k}}V_{\bo{k}}^*$ is a projection.
In particular this holds for all $a \in I_{\un{x}}$ and the proof is complete.
\end{proof}

\subsubsection{Injective systems}

For injective systems the picture is rather clear.

\begin{proposition}
Let $\al \colon \bZ_+^n \to \End(A)$ be a unital C*-dynamical system and $\si \colon \bR \to \Aut(\N\T(A,\al))$ be the action related to $\un{\la} \in \bR^n\setminus \{\un{0}\}$.
If $\al$ is injective then there are no KMS states on $\N\O(A,\al)$ for $\be \neq 0$.
\end{proposition}
\begin{proof}
Recall that injectivity of $\al$ is equivalent to the $U_\Bi$ being unitaries by Proposition \ref{P: injectivity}.
Let $\psi$ be a $(\si, \be)$-KMS state on $\N\O(A,\al)$.
As in the first part of the proof of Proposition \ref{P: cond KMS} we can have an estimation for the $\la_k\be$ by using the $U_{\bo{k}}$ in the place of $V_{\bo{k}}$.
However now $U_{\bo{k}}$ is a unitary and we obtain equality, i.e. $1 = e^{-\la_k\be}$, thus $\la_k \be=0$ for all $k=1, \dots, n$.
Since $\si$ is not trivial there is a $\la_k \neq 0$ hence $\be=0$.
\end{proof}

Therefore if $\al \colon \bZ_+^n \to \End(A)$ is an injective C*-dynamical system then the only possible KMS states are the tracial states on $\N\O(A,\al)$.
Recall that in this case $\N\O(A,\al) \simeq \wt{A} \rtimes_{\wt{\al}} \bZ^n$ for the automorphic extension $\wt{\al} \colon \bZ^n \to \Aut(\wt{A})$.
When $\wt{A}$ admits tracial states then $\N\O(A,\al) \simeq \wt{A} \rtimes_{\wt{\al}} \bZ^n$ admits tracial states that appear as evaluations on the $(\un{0},\un{0})$-entry.

\subsubsection{Non-injective systems}

For the non-injective case we will use the analysis of the KMS states on $\N\T(A,\al)$.

\begin{theorem}\label{T: KMS CNP}
Let $\al \colon \bZ_+^n \to \End(A)$ be a unital C*-dynamical system, $\si \colon \bR \to \Aut(\N\T(A,\al))$ be the action related to $\un{\la} \in \bR^n$, and $\be \in \bR$ such that $\la_k \be >0$ for all $k=1, \dots, n$.
Then there is an affine homeomorphism $\tau \mapsto \phi_\tau$ from the simplex of the tracial states on $A$ that vanish on $I_{\un{1}}$ onto the simplex of the $(\si,\be)$-KMS states on $\N\O(A,\al)$ such that
\begin{align*}
\phi_\tau(U_{\un{x}} a U_{\un{y}}^*)
=
\de_{\un{x}, \un{y}} \cdot
e^{-\sca{\un{x},\un{\la}}\be} \cdot \prod_{i=1}^n (1 - e^{-\la_i\be}) \cdot
\sum_{\un{w} \in \bZ_+^n} e^{-\sca{\un{w},\un{\la}}\be} \tau\al_{\un{w}}(a),
\end{align*}
for all $a\in A$, $\un{x}, \un{y} \in \bZ_+^n$.
\end{theorem}
\begin{proof}
Recall that $\N\O(A,\al)$ is a quotient of $\N\T(A,\al)$ and that we have already obtained a parametrization of the KMS states on $\N\T(A,\al)$.
Thus it suffices to show that a tracial state $\tau$ on $A$ vanishes on $I_{\un{1}}$ if and only if $\psi_\tau$ of Proposition \ref{P: constr KMS} defines a KMS state $\phi_\tau$ on $\N\O(A,\al)$ by $\phi_\tau q = \psi_\tau$ for the quotient map $q \colon \N\T(A,\al) \to \N\O(A,\al)$.
Since the action $\si$ respects the quotient mapping, this is equivalent to showing that $\psi_\tau$ vanishes on the elements $a \cdot \prod_{i \in \supp \un{y}} (I - V_\Bi V_\Bi^*)$ for all $a \in I_{\un{y}}$, and for all $\un{y} > 0$.

To this end fix $\un{y} \in \bZ_+^n$.
Since $I_{\un{y}} = I_{\un{y} \wedge \un{1}}$, then without loss of generality we may assume that $\un{y} \leq \un{1}$ with support $\supp \un{y} = \{\bo{1}, \dots, \bo{d}\}$.
Then a computation similar to that of Theorem \ref{T: KMS TNP} gives that
\[
\psi_\tau(a \cdot \prod_{i=1}^d (I - V_{\Bi} V_{\Bi^*}))
=
\prod_{i=1}^n (1 - e^{-\be_i})^{-1} \tau(a).
\]
The proof follows in exactly the same way by replacing $\prod_{i=1}^n (I - V_{\Bi} V_{\Bi^*})$ with $\prod_{i=1}^d (I - V_{\Bi} V_{\Bi^*})$ and $\un{1}$ with $\un{y}$.
Recall here that Lemma \ref{L: combinatoric} holds for any $\un{y} \leq \un{1}$ .

For the forward implication fix $\tau$ be a state on $A$ that vanishes on $I_{\un{1}}$.
Since $I_{\un{y}} \subseteq I_{\un{1}}$ for all $\un{y} > 0$, then $\tau$ vanishes on $I_{\un{y}}$ as well.
The above equation then shows that $\psi_\tau$ vanishes on the ideals generated by $a \cdot \prod_{\Bi \in \supp \un{y}} (I - V_\Bi V_\Bi^*)$ for all $a \in I_{\un{y}}$.

Conversely if $\phi$ is a $(\si,\be)$-KMS state on $\N\O(A,\al)$ then $\phi q$ is a $(\si,\be)$-KMS state on $\N\T(A,\al)$ where $q \colon \N\T(A,\al) \to \N\O(A,\al)$ is the quotient $*$-epimorphism.
Hence we obtain that $\phi q = \psi_\tau$ for some tracial state $\tau$ of $A$, by Theorem \ref{T: KMS TNP}.
Therefore $\psi_\tau$ vanishes on $a \cdot \prod_{i=1}^n (I - V_{\Bi} V_{\Bi^*})$ for all $a \in I_{\un{1}}$.
Then the above computation shows that $\tau$ vanishes on $I_{\un{1}}$, and the proof is complete.
\end{proof}

\subsubsection{Ground states and KMS${}_\infty$ states}

Similar computations give the following analogues for the ground states and the KMS${}_\infty$ states on $\N\O(A,\al)$. In particular for the KMS${}_\infty$ states we make use of Theorem \ref{T: KMS CNP}.

\begin{corollary}\label{C: ground states CNP}
Let $\al \colon \bZ_+^n \to \End(A)$ be a unital C*-dynamical system.
The map $\tau \mapsto \vpi_\tau$ with
\[
\vpi_\tau(U_{\un{x}} a U_{\un{y}}^*) =
\begin{cases}
\tau(a) & \text{ for } \un{x}=\un{0}=\un{y},\\
0 & \text{ otherwise},
\end{cases}
\]
is an affine homeomorphism from the state space $S(A)$ (resp. tracial state space $T(A)$) onto the ground states (resp. KMS${}_\infty$ states) of $\N\O(A,\al)$.
\end{corollary}

\subsubsection{Applications}

As in the case of $\N\T(A,\al)$, our analysis can be used to treat the particular cases when $\la_k=1$ for all $k=1, \dots, n$, or when $\la_k=0$ for some $k=1, \dots, n$.

\begin{example}
In the case where $\la_1 = \dots = \la_n = 1$, let the action $\si \colon \bR \to \Aut(\N\O(A,\al))$ be given by $\si_t = \ga_{(\exp(it), \dots, \exp(it))}$.
Then for every tracial state $\tau$ of $A$ that vanishes on $I_{\un{1}}$ and $\be>0$ there is a $(\si,\be)$-KMS state $\psi_\tau$ of $\N\O(A,\al)$ such that
\[
\psi_\tau(U_{\un{x}} a U_{\un{y}}^*)
 =
\de_{\un{x}, \un{y}} \cdot
e^{-|\un{x}|\be} \cdot (1 - e^{-\be})^n
\sum_{\un{w} \in \bZ_+^n} e^{-|\un{w}|\be} \tau\al_{\un{w}}(a),
\]
for all $a\in A$ and $\un{x}, \un{y} \in \bZ_+^n$.
This representation is a parametrization of the $(\si,\be)$-KMS states.
\end{example}

Now let us examine the case where some $\la_k$ are zeroes.
Let us assume that $\la_{d+1} = \dots = \la_n =0$ and $\la_1, \dots, \la_d \neq 0$.
We aim to show that the $(\si,\be)$-KMS states on $\N\O(A,\al)$ are determined by the $(\si',\be)$-KMS states on $\N\O(\Om,\om)$ for a suitably chosen C*-dynamical system $\om \colon \bZ_+^d \to \End(\Om)$ and for
\[
\si_t' = \ga_{(\exp(i\la_1t), \dots, \exp(i\la_d t))}'.
\]
To this end let the C*-algebra
\[
\Om= \ol{\spn}\{U_{\un{x}} a U_{\un{y}}^* \mid a\in A, \supp \un{x}, \supp \un{y} \subseteq \{\bo{d+1}, \dots, \bo{n}\} \}.
\]
Observe that $\Om$ is a C*-subalgebra of $\N\O(A,\al)$.
For $\Bi = \{\bo{1}, \dots, \bo{d}\}$ we can extend the $*$-endomorphism $\al_\Bi$ to a $*$-endomorphism $\om_{\Bi}$ of $\Om$ defined by
\[
\om_{\Bi} (f) = U_{\Bi}^* f U_{\Bi}, \foral f \in \Om.
\]
Then $\si \colon \bR^n \to \Aut(\N\O(A,\al))$ defines an action $\si' \colon \bR^d \to \Aut(\N\O(\Om,\om))$.

\begin{proposition}\label{P: d-states CNP}
With the aforementioned notation, a state $\vpi$ is a $(\si, \be)$-KMS state on $\N\O(A,\al)$ if and only if $\vpi$ is a $(\si',\be)$-KMS state on $\N\O(\Om,\om)$.
\end{proposition}
\begin{proof}
The proof follows in the same way as in Proposition \ref{P: d-states TNP} once we show that $\N\O(\Om,\om) = \N\O(A,\al)$.
To this end we will show that the representation $(\id_\Om, U)$ defines a faithful Cuntz-Nica covariant pair of $\N\O(\Om,\om)$.
It is evident that $\id_\Om$ is injective on $\Om$ and that $(\id_\Om, U)$ admits a gauge action inherited from the gauge action of $\N\O(A,\al)$.
It remains to show that $(\id_\Om, U)$ is also a Cuntz-Nica covariant pair.

\medskip

\noindent \textbf{Claim.} Let $\un{w} \in \bZ_+^d$.
Then we have that $f \in (\bigcap_{\bo{i} \in \supp \un{w}} \ker\om_{\Bi})^\perp \subseteq \Om$ if and only if $f \cdot \prod_{\Bi \in \supp \un{w}} (I - U_{\Bi}U_{\Bi}^*) = 0$.

\medskip

\noindent \textbf{Proof of the Claim.} For convenience let $x = \prod_{\Bi \in \supp \un{w}} (I - U_{\Bi}U_{\Bi}^*)$.
Without loss of generality assume that $\supp \un{w} = \{\bo{1}, \dots, \bo{m}\}$ with $m \leq d$.
Let $g \in \bigcap_{\bo{i} \in \supp \un{w}} \ker\om_{\Bi}$ and $f \in \Om$ such that $f x= 0$.
Then we get that
\begin{align*}
0
& =
f x \cdot g \\
& =
f \cdot \Big( \prod_{i=1}^{m-1} (I - U_{\Bi}U_{\Bi}^*) \Big) \cdot (I - U_{\bo{m}}U_{\bo{m}}^*) g \\
& =
f \cdot \Big( \prod_{i=1}^{m-1} (I - U_{\Bi}U_{\Bi}^*) \Big) \cdot (g - U_{\bo{m}} \om_{\bo{m}}(g) U_{\bo{m}}^*) \\
& =
f \cdot \prod_{i=1}^{m-1} (I - U_{\Bi}U_{\Bi}^*) \cdot g,
\end{align*}
since $g \in \ker \om_{\bo{m}}$.
Inductively we obtain that $f g = 0$, which shows that $f$ is in $(\bigcap_{\bo{i} \in \supp \un{w}} \ker\om_{\Bi})^\perp$.

Conversely suppose that $f \in (\bigcap_{\bo{i} \in \supp \un{w}} \ker\om_{\Bi})^\perp$.
Note that $x$ is a projection that commutes with $\Om$.
Let us form the C*-algebra
\[
\fA = \ca(\Om, x) = \Om x + \Om(1-x)
\]
inside $\N\O(A,\al)$.
We will denote by $\wt\om_{\Bi}$ the extension of $\om_{\Bi}$ to $\fA$.
Then we get that
\[
x \in \bigcap_{\Bi \in \supp \un{w}} \ker\wt\om_{\Bi} \subseteq \fA
\]
as well.
As a consequence we have that
\[
fx \in \bigcap_{\Bi \in \supp \un{w}} \ker\wt\om_{\Bi} \subseteq \fA.
\]
On the other hand every $\wt{g} \in \fA$ attains an orthogonal decomposition $\wt{g} = g x + \wt{g} (1-x)$ for some $g \in \Om$.
We can then write
\[
g = gx + g(1-x) = \wt{g} - \wt{g}(1-x) + g(1-x).
\]
Let $\wt{g} \in \bigcap_{\Bi \in \supp \un{w}} \ker\wt\om_{\Bi}$.
Then we obtain
\begin{align*}
\om_{\Bi}(g)
& =
U_\Bi^* g U_{\Bi} \\
& =
U_{\Bi}^* \wt{g}U_{\Bi} - U_{\Bi}^* \wt{g} (1-x) U_{\Bi} +  U_{\Bi}^*g(1-x)U_{\Bi} \\
& =
\wt\om_\Bi(\wt{g}) + 0 = 0,
\end{align*}
since $(1-x) U_\Bi = 0$ for all $\Bi \in \supp \un{w}$.
Therefore we have that $g \in \bigcap_{\Bi \in \supp \un{w}} \ker\om_\Bi \subseteq \Om$.
As a consequence we obtain that
\[
f x \wt{g} = f x g x + fx\wt{g}(1-x)= x fg x = 0,
\]
since $f$ is perpendicular to $\bigcap_{\Bi \in \supp \un{w}} \ker\om_\Bi$.
Thus $fx \in (\bigcap_{\Bi \in \supp \un{w}} \ker\wt\om_{\Bi})^\perp$ as well.
This shows that $fx=0$ and the proof of the claim is complete.

\medskip

For distinction let us denote by $J_{\un{w}}$ the ideals of $\om \colon \bZ_+^d \to \End(\Om)$ with $\supp \un{w} = \{\bo{1}, \dots, \bo{d} \}$.
Now let $f \in J_{\un{w}}$ for $\un{w} \in \bZ_+^d$ with $\supp \un{w} = \{\bo{1}, \dots, \bo{m}\}$.
Then in particular we get that $f \in (\bigcap_{\bo{i} \in \supp \un{w}} \ker\om_{\Bi})^\perp$.
By the claim we obtain
\[
\id_\Om(f) \cdot \prod_{\Bi \in \supp \un{w}} (I - U_{\Bi}U_{\Bi}^*) = f \cdot \prod_{\Bi \in \supp \un{w}} (I - U_{\Bi}U_{\Bi}^*) = 0
\]
which completes the proof.
\end{proof}

\section{Appendix}\label{S: app}

Multivariable systems over $\bZ_+^n$ admit multivariable gauge actions $\si$ of $\bR^n$.
Surprisingly all the computations of Section \ref{S: KMS} still hold for the action
\[
\si_{\un{t}} = \ga_{(\exp(it_1), \dots, \exp(it_n))}, \text{ for } \un{t} = (t_1, \dots, t_n) \in \bR^n,
\]
by substituting $\la_k \be$ with any $\be_k \in \bR$.
So one may ask why not apply the same analysis in the multivariable context.
That is, for an action $\si \colon \bR^n \to \Aut(\A)$ of a C*-algebra $\A$ and for $\un{0} < \un{\be} \in \bR^n$, one may consider the condition
\[
\psi(ab) = \psi(b \si_{i \un{\be}}(a)),
\]
for all $a, b$ in a norm dense $*$-subalgebra of analytic elements in $\A$.
However the purpose of the KMS condition is to build a specific analytic function $F_{a,b}$ on an (unbounded) domain.
In the one-variable case this is achieved by using the Phragm\'{e}n-Lindel\"{o}f principle for strips.
The application of this idea for tubes in $\bC^n$ requires the insertion of new data.
Here we present a multivariable KMS condition subject to a prescribing set.
The required proofs will follow after this exposition.

\begin{definition}
Let $\un{0} < \un{\be} \in \bR^n$.
A \emph{$\un{\be}$-prescribing set $\La_{\un{\be}}$} is a subset of
\[
C_{\un{\be}}:= \{\un{\ga} = \sum_{k=1}^n \eps_k \be_k \, \bo{k} \mid \eps_k=0,1 \}
\]
such that $\un{0} \notin \La_{\un{\be}}$.
\end{definition}

For the definition of a multivariable KMS state we will require all $3$ data $(\si,\un{\be},\La_{\un{\be}})$.

\begin{definition}
Let $\si\colon \bR^n \to \Aut(\A)$ be a C*-dynamical system and $\La_{\un{\be}}$ be a prescribing set for $\un{0} < \un{\be} \in \bR^n$.
A state $\tau$ of $\A$ is called a \emph{$(\si,\La_{\un{\be}})$-KMS state on $\A$} if for every pair of elements $a,b \in \A$ there is a complex-valued function $F_{a,b}$ that is analytic on the interior of
\[
D = \{z \in \bC^n \mid 0 \leq \im(z_k) \leq \be_k, k=1, \dots, n\}
\]
and continuous and bounded on $D$ such that $F_{a,b}(\un{t}) = \tau(a \si_{\un{t}}(b)$
and
\begin{align*}
F_{a,b}(\un{t} + i \un{\ga}) =
\begin{cases}
\tau(\si_{\un{t}}(b)a) & \text{ when } \un{\ga} \in \La_{\un{\be}}, \\
\tau(a \si_{\un{t}}(b)) & \text{ when } \un{\ga} \notin \La_{\un{\be}},
\end{cases}
\end{align*}
for all $\un{\ga} \in C_{\un{\be}}$, and for all $\un{t} \in \bR^n$.
\end{definition}

The first objective is to show that there are enough analytic elements inside $\A$ for an action $\si$ of $\bR^n$.

\begin{proposition}\label{P: analytic}
Let $\si\colon \bR^n \to \Aut(\A)$ be a C*-dynamical system and let $\tau$ be a state on $\A$.
Then there is a norm-dense $\si$-invariant $*$-subalgebra $\A_{\textup{an}}$ of $\A$ such that for every $a \in \A_{\textup{an}}$ the function $\un{t} \mapsto \si_{\un{t}}(a)$ can be analytically continued to an entire function on $\bC^n$.
\end{proposition}

For such analytic elements $a$ in $\A$ we write $\si_{\un{z}}(a) := f_a(\un{z})$. We can then characterize the $(\si,\La_{\un{\be}})$-KMS states on $\A$ by the following \emph{$(\si,\La_{\un{\be}})$-KMS condition}.

\begin{theorem}\label{T: KMS char}
Let $\si\colon \bR^n \to \Aut(\A)$ be a C*-dynamical system and let $\tau$ be a state on $A$.
Let $\La_{\un{\be}}$ be a prescribing set for $\un{0} < \un{\be} \in \bR^n$.
Then $\tau$ is a $(\si,\La_{\un{\be}})$-KMS state if and only if for all $\un{\ga} \in C_{\un{\be}}$ we have that
\[
\tau(a \si_{i\un{\ga}}(b)) =
\begin{cases}
\tau(b a) & \text{ when } \un{\ga} \in \La_{\un{\be}}, \\
\tau(ab) & \text{ when } \un{\ga} \notin \La_{\un{\be}},
\end{cases}
\]
for all $a,b$ in a norm-dense $\si$-invariant $*$-subalgebra of $\A_{\text{an}}$.
\end{theorem}

Even though this condition is necessary and sufficient for building analytic functions on tubes, it is rather strong.
In the next example we show that the Toeplitz-Nica-Pimsner algebras may not attain such states.

\begin{example}\label{E: non-example}
Let $\N\T(\bZ^2_+)$ be the Toeplitz-Nica-Pimsner of the trivial system $\id \colon \bZ^2_+ \to \End(\bC)$.
This is the universal C*-algebra generated by two doubly commuting isometries, say $V_{(1,0)}$ and $V_{(0,1)}$.
We can extend $V$ to a semigroup action $V \colon \bZ_+^2 \to \B(H)$ by isometries.
Then the monomials $V_{(x_1, x_2)} V_{(y_1, y_2)}^*$ span a dense subset of $\N\T(\bZ^2_+)$.
Moreover $\N\T(\bZ^2_+)$ admits a gauge action $\{\ga_{(z_1,z_2)}\}_{(z_1, z_2) \in \bT^2}$.
Let the group action $\si \colon \bR^2 \to \Aut(\N\T(\bZ^2_+))$ defined by
\[
\si_{(t_1,t_2)} = \ga_{(\exp(it_1), \exp(it_2))}.
\]
We will show that $\N\T(\bZ^2_+)$ does not admit $(\si,\La_{(\be_1, \be_2)})$-KMS states on $(\be_1, \be_2) \neq (0,0)$ for any prescribing set $\La_{(\be_1, \be_2)}$, with the trivial exception when $\be_1=0$ or $\be_2 =0$.
In this case there is reduction to the one-variable case.

First suppose that the prescribing set contains $\{(\be_1,0), (0,\be_2)\}$.
Then a $(\si,\La_{(\be_1, \be_2)})$-KMS state $\psi$ would satisfy
\[
\psi(V_{(1,0)} V_{(1,0)}^*) = \psi(V_{(1,0)}^* \si_{(i\be_1,0)}(V_{(1,0)})) = e^{-\be_1}.
\]
On the other hand we get that
\[
\psi(V_{(1,0)} V_{(1,0)}^*) = \psi(V_{(1,0)}^* \si_{(0,i\be_2)}(V_{(1,0)})) = \psi(V_{(1,0)}^* V_{(1,0)}) =1.
\]
Therefore $\be_1 =0$ and similarly $\be_2 =0$.

Next assume that $\La_{(\be_1, \be_2)} = \{(\be_1,\be_2)\}$.
Then we obtain
\[
1 = \psi(V_{(1,0)}^* V_{(1,0)}) = \psi(V_{(1,0)}^* \si_{(i\be_1,0)}(V_{(1,0)})) = e^{-\be_1},
\]
hence $\be_1 = 0$ and similarly $\be_2=0$.

Finally assume that $\La_{(\be_1, \be_2)}$ contains $(\be_1,0)$ and does not contain $(0,\be_2)$.
Then it follows that $\La_{(\be_1, \be_2)} = \{(\be_1,0), (\be_1,\be_2)\}$, since
\begin{align*}
\psi(V_{\un{x}} V_{\un{y}}^* V_{\un{z}} V_{\un{w}}^*)
& =
\psi(V_{\un{z}} V_{\un{w}}^* \si_{(i\be_1,0)}(V_{\un{x}} V_{\un{y}}^*)) \\
& =
\psi(V_{\un{z}} V_{\un{w}}^* \si_{(0,i\be_2)}\si_{(i\be_1,0)}(V_{\un{x}} V_{\un{y}}^*)) \\
& =
\psi(V_{\un{z}} V_{\un{w}}^* \si_{i(\be_1,\be_2)}(V_{\un{x}} V_{\un{y}}^*)).
\end{align*}
It is evident that the $(\si,\La_{(\be_1, \be_2)})$-KMS states in this case are in bijection with the $(\si',\be_1)$-KMS states with respect to the one-variable action $\si'_t = \si_{(t,0)}$.
Such states exist, but they are trivial in the sense that they do not give non-trivial multivariable states.
The case when $\La_{(\be_1, \be_2)}$ contains $(0,\be_2)$ is settled in the dual way.
\end{example}

Let us finish this off with the required proofs.
In what follows we use the simplified multivariable notation.
We write $t \equiv \un{t} \in \bC^n$ and we use the symbol $t_i$ to denote either the i-th coordinate of $t$ or the vector $t_i \cdot \Bi \in \bC^n$.
The difference will be clear by the context.

\subsection{Proof of Proposition \ref{P: analytic}}

For this subsection fix an $a \in \A$ and define the sequence $(a_m)$ by
\begin{align*}
a_m
& =
\sqrt{\frac{m}{\pi}}^{\, n} \cdot \int_{\bR^n} \si_t(a) e^{-m \sum_{k=1}^n t_k^2} dt \\
& =
\sqrt{\frac{m}{\pi}}^{\, n} \cdot \int_{\bR} \cdots \int_{\bR} \si_t(a) e^{-mt_1^2} \cdots e^{-mt_n^2} dt_1 \dots dt_n.
\end{align*}
By \cite[Proposition 2.5.18]{BraRob87} the one-variable integral $\int_\bR \si_{t_1}(a) e^{-nt_1^2} dt_1$ defines an element in $\A$ for $\si_{t_1} \equiv \si_{(t_1,0,\dots,0)}$.
By using induction and the formula
\[
\int_{\bR} \si_t(a) e^{-mt_1^2} \cdots e^{-mt_n^2} dt_1
=
\si_{t - t_1} \left( \int_\bR \si_{t_1}(a) e^{-mt_1^2} dt_1\right) e^{-mt_2^2} \cdots e^{-mt_n^2},
\]
we deduce that the elements $a_m$ are well defined.
Let $f_m \colon \bC^n \to A$ be the function
\begin{align*}
f_m(z)
& =
\sqrt{\frac{m}{\pi}}^{\, n} \cdot \int_{\bR^n} \si_t(a) e^{-m \sum_{k=1}^n (t_k - z_k)^2} dt \\
& =
\sqrt{\frac{m}{\pi}}^{\, n} \cdot \int_{\bR} \cdots \int_{\bR} \si_t(a) e^{-m(t_1-z_1)^2} \cdots e^{-m(t_n-z_n)^2} dt_1 \dots dt_n.
\end{align*}
The function $t_k \mapsto e^{m(t_k-z_k)^2}$ is in $L^1(\bR)$, hence by induction on \cite[Proposition 2.5.18]{BraRob87} we get that $f_m(z)$ is a well defined element in $\A$.
For $z = s \in \bR^n$ we compute
\begin{align*}
f_m(s)
& =
\sqrt{\frac{m}{\pi}}^{\, n} \cdot \int_{\bR} \cdots \int_{\bR} \si_t(a) e^{-m(t_1-s_1)^2} \cdots e^{-m(t_n-s_n)^2} dt_1 \dots dt_n \\
& =
\sqrt{\frac{m}{\pi}}^{\, n} \cdot \int_{\bR} \cdots \int_{\bR} \si_{t+s}(a) e^{-mt_1^2} \cdots e^{-mt_n^2} dt_1 \dots dt_n \\
& =
\si_s\left( \sqrt{\frac{m}{\pi}}^{\, n} \cdot \int_{\bR} \cdots \int_{\bR} \si_{t}(a) e^{-mt_1^2} \cdots e^{-mt_d^2} dt_1 \dots dt_n \right) \\
& =
\si_s(a_m).
\end{align*}
Therefore the mapping $s \mapsto \si_s(a_m)$ extends to a well defined function $f_m \colon \bC^n \to \A$.
We will show that $f_m$ is entire analytic, i.e. the limit
\begin{align*}
\lim_{h_k \to 0} h_k^{-1} [f_m(z+h_k) - f_m(z)]
\end{align*}
exists for all $z \in \bC^n$ and $k=1, \dots, n$.
For $k=1$ we calculate
\begin{align*}
f_m(z+h_1) - f_m(z)
& =
\sqrt{\frac{m}{\pi}}^{\, n} \int_{\bR} \int_{\bR^{n-1}} x_a(t) (e^{-m(t_1-z_1-h_1)^2} - e^{-mz_1^2}) \wh{dt_1}dt_1,
\end{align*}
where $\wh{dt_1} = dt_2 \cdots dt_n$ and
\begin{align*}
x_a(t)
=
\si_t(a) e^{-mt_2^2} \cdots e^{-mt_n^2}
=
\si_{t_1} \left(\si_{t-t_1}(a) e^{-mt_2^2} \cdots e^{-mt_d^2} \right) =\si_{t_1}(a').
\end{align*}
Therefore we have
\begin{align*}
f_m(z+h_1) - f_m(z)
& =
\sqrt{\frac{m}{\pi}} \int_{\bR} \si_{t_1}(a'')(e^{-m(t_1-z_1-h_1)^2} - e^{-mz_1^2})dt_1,
\end{align*}
where
\[
a'' = \sqrt{\frac{m}{\pi}}^{\, n-1} \int_{\bR^{n-1}} \si_{t-t_1}(a) e^{-mt_2^2} \cdots e^{-mt_n^2} \wh{dt_1},
\]
and thus $\nor{a''} \leq \nor{a}$.
Consequently it suffices to show that the limit
\[
\lim_{h_1 \to 0} h_1^{-1}[\int_{\bR} \si_{t_1}(a'')(e^{-m(t_1-z_1-h_1)^2} - e^{-mz_1^2})dt_1]
\]
exists in $\A$.
This follows from \cite[Proposition 2.5.22]{BraRob87}.

It remains to show that the sequence $(a_m)$ converges to $a$ in norm.
To this end we compute
\begin{align*}
a_m - a
& =
\sqrt{\frac{m}{\pi}}^{\, n}  \int_{\bR^n} (\si_t(a) - a) e^{-m \sum_{k=1}^n t_k^2} dt.
\end{align*}
For $\eps>0$ let $\de>0$ such that $\nor{\si_t(a) - a} <\eps/2$ for all $|t| < \de$.
Moreover for the same $\eps>0$ let $m_0 \in \bZ_+$ such that
\[
\sqrt{\frac{m}{\pi}}^{\, n} \cdot \int_{|t| \geq \de} \si_t(a) e^{-m \sum_{k=1}^n t_k^2} dt \leq \frac{\eps}{4 \nor{a}},
\]
for all $m \geq m_0$.
Then for $m \geq m_0$ we get that
\begin{align*}
\nor{a_m - a}
& =
\sqrt{\frac{m}{\pi}}^{\, n} \cdot \int_{|t| \leq \de} \nor{\si_t(a) - a} e^{-m \sum_{k=1}^n t_k^2} dt \, + \\
& \hspace{4cm} +
\sqrt{\frac{m}{\pi}}^{\, n} \cdot \int_{|t| \geq \de} \nor{\si_t(a) - a} e^{-m \sum_{k=1}^n t_k^2} dt \\
& \leq
\frac{\eps}{2} \sqrt{\frac{m}{\pi}}^{\, n} \cdot \int_{|t| \leq \de} e^{-m \sum_{k=1}^n t_k^2} dt
 +
2 \nor{a} \sqrt{\frac{m}{\pi}}^{\, n} \cdot \int_{|t| \geq \de} e^{-m \sum_{k=1}^n t_k^2} dt \\
& \leq
\frac{\eps}{2} \sqrt{\frac{m}{\pi}}^{\, n} \cdot \sqrt{\frac{\pi}{m}}^{\, n} + 2 \nor{a} \frac{\eps}{4 \nor{a}} = \eps,
\end{align*}
which finishes the proof. \hfill{\qed}

\subsection{Extending the action}

The action $\si$ extends pointwise on $\A_{\text{an}}$ to an action of $\bC^n$ in the sense that
\[
\si_z \si_w(a) = \si_{z+w}(a) \foral a \in \A_{\text{an}} \text{ and } z,w \in \bC^n.
\]
Indeed an element
\[
a = \sqrt{\frac{m}{\pi}}^{\, n} \int_{\bR^n} \si_t(b) e^{-m\sum_{k=1}^n t_k^2} dt \; \in \; \A_{\textup{an}}
\]
for some $b \in \A$.
Then the function $\bR^n \ni r \mapsto \si_r \si_w(a)$ is the restriction of the entire analytic function $z \mapsto \si_z \si_w(a)$.
For the latter first compute
\begin{align*}
\si_r \si_w(a)
& =
\si_r \Big( \sqrt{\frac{m}{\pi}}^{\, n} \int_{\bR^n} \si_t(b) e^{-m\sum_{k=1}^n (t_k - w_k)^2} dt \Big) \\
& =
\sqrt{\frac{m}{\pi}}^{\, n} \int_{\bR^n} \si_{r+t}(b) e^{-m\sum_{k=1}^n (t_k - w_k)^2} dt \\
& =
\sqrt{\frac{m}{\pi}}^{\, n} \int_{\bR^n} \si_{t}(b) e^{-m\sum_{k=1}^n (t_k - r_k - w_k)^2} dt.
\end{align*}
Hence the function $r \mapsto \si_r\si_w(a)$ extends to the entire function
\[
z \mapsto \sqrt{\frac{m}{\pi}}^{\, n} \int_{\bR^n} \si_t(b) e^{-m\sum_{k=1}^n (t_k - z_k - w_k)^2} dt,
\]
which implies that $\si_w(a)$ is an analytic element.
Thus it has an analytic extension given by $z \mapsto \si_z \si_w(a)$.
However a second analytic extension is given by
\[
z \mapsto \si_{z+w}(a) := \sqrt{\frac{m}{\pi}}^{\, n} \int_{\bR^n} \si_{t}(b) e^{-m\sum_{k=1}^n (t_k - z_k - w_k)^2} dt.
\]
Uniqueness of the analytic extension then shows that $\si_z \si_w (a) = \si_{z+w}(a)$.

\subsection{Invariance of KMS states}

For an analytic element $a$ in $\A$ we write $\si_z(a) := f_a(z)$.
The $(\si,\La_{\be})$-KMS condition for a state $\tau$ of $A$ is then
\[
\tau(a f_b(i\ga)) = \tau(a \si_{i\ga}(b)) =
\begin{cases}
\tau(ba) & \text{ when } \ga \in \La_{\be}, \\
\tau(ab) & \text{ when } \ga \notin \La_{\be},
\end{cases}
\]
for all $a,b$ in a norm-dense $*$-subalgebra of analytic elements of $\A$.

\begin{proposition}
Let $\tau$ be a state on $\A$ that satisfies the $(\si,\La_{\be})$-KMS condition.
Then $\tau$ is $\si$-invariant.
\end{proposition}
\begin{proof}
Let $\be = (\be_1, \dots, \be_n)$ in $\bR^n$.
Without loss of generality we may assume that $\be_k \neq 0$ for all $k = 1, \dots, n$.
Otherwise we pass to the subsystem on $\bR^{n-d}$ where $d$ is the number of the $\be_k$ that are zero.

Fix $a$ be an element in a norm-dense $*$-subalgebra of $\A_{\text{an}}$ given by the definition of the KMS condition.
Define the function $F(t) := \tau\si_t(a)$.
Then $F(t)$ extends to an entire analytic function $F(z) = \tau\si_z(a)$.
Fix attention to the first co-ordinate and let
\[
a_1 = \si_{z-z_1}(a) \qand F_1( \cdot ) = F(\cdot, z_2, \dots, z_d).
\]
Since
\[
|F_1(z_1)| \leq \nor{\si_{z_1}(a_1)} = \nor{\si_{\re(z_1)} \si_{i \im(z_1)}(a_1)} \leq \nor{\si_{i \im(z_1)}(a_1)}
\]
we get that $F_1$ is bounded on the strip $D_1 = \{ z_1 \in \bC \mid 0 \leq \im{z_1} \leq \be_1 \}$ by
\[
M = \sup\{ \nor{\si_{it_1}(a_1)} \mid t_1 \in [0, \be_1] \}.
\]
Then $M$ is finite since the function
\[
[0,\be_1] \ni t_1 \mapsto \nor{\si_{it_1}(a_1)}
\]
is continuous.
Let $(e_j)$ be an approximate identity of $\A$ and compute
\begin{align*}
F_1(z_1 + i\be_1)
& =
\lim_j \tau(e_j \, \si_{i \be_1} \si_{z_1}(a_1)) = \lim_j \tau(\si_{z_1}(a_1) \, e_j)
  =
F(z_1).
\end{align*}
Thus $F_1$ is periodic, hence bounded by $M$.
By Liouville's Theorem then $F_1$ is constant.
Hence we have that $\partial F / \partial z_1 = \partial F_1 / \partial z_1 =0$, therefore $F(z_1, \dots, z_d) = F(z_2, \dots, z_d)$.
Inductively we get that $F$ is constant in $\bC^n$.
Therefore we obtain
\[
\tau(a) = F(0) = F(t) = \tau\si_t(a)
\]
for all $a$ in a norm-dense $*$-subalgebra of $\A_{\text{an}}$.
Since $\A_{\text{an}}$ is dense in $\A$ we get that $\tau(a) = \tau \si_t(a)$ for all $a\in \A$.
\end{proof}

\subsection{Proof of Theorem \ref{T: KMS char}}

First suppose that $\tau$ is a state on $\A$ such that
\[
\tau(a \si_{i\ga}(b)) =
\begin{cases}
\tau(b a) & \text{ when } \ga \in \La_{\be}, \\
\tau(ab) & \text{ when } \ga \notin \La_{\be},
\end{cases}
\]
for all $a,b$ in a norm-dense $*$-subalgebra of $\A_{\text{an}}$.
For such fixed $a,b \in \A_{\text{an}}$ let the function
\[
\bC^n \ni z \mapsto F_{a,b}(z) = \tau(a \si_z(b)).
\]
Then $F_{a,b}$ is entire analytic and
\begin{align*}
F_{a,b}(t + i \ga) =
\begin{cases}
\tau(\si_t(b)a) & \text{ when } \ga \in \La_{\be}, \\
\tau(a \si_t(b)) & \text{ when } \ga \notin \La_{\be},
\end{cases}
\end{align*}
for all $\ga \in C_{\be}$.
The function $\bC^n \ni z \mapsto \si_z(b)$ is analytic hence the function
\[
\{s \in \bR^n \mid 0 \leq s_k \leq \be_k\} \to \bR : s \mapsto \nor{\si_{is}(b)}
\]
is continuous, thus bounded.
For
\[
M = \sup\{ \nor{\si_{is}(b)} \mid 0 \leq s_k \leq \be_k \}
\]
we get that
\[
|F_{a,b}(t+is)| = |\tau(a \si_t \si_{is}(b))| \leq M \cdot \nor{a}
\]
for all $t + is \in D$.

Now we pass to the general case where $a,b \in A$.
To this end let $a_m$ and $b_m$ in the dense subalgebra of $\A_{\textup{an}}$ with $\nor{a_m} \leq \nor{a}$ and $\nor{b_m} \leq \nor{b}$ such that $a= \lim_m a_m$ and $b =\lim_m b_m$, and define $F_m(z) = F_{a_m,b_m}(z)$.
Our aim is to show that the sequence $(F_m)$ is uniformly Cauchy so that defining $F_{a,b}$ as the limit of $F_m$ gives rise to a continuous and bounded function on $D$.
Moreover we will eventually have that
\begin{align*}
F_{a,b}(t+ i \ga)
& =
\lim_m \tau(a_m\si_t(b_m))
=
\tau(a\si_t(b)),
\end{align*}
when $\ga \notin \La_{\be}$, and
\begin{align*}
F_{a,b}(t+ i \ga)
& =
\lim_m F_{a_m,b_m}(t + i\ga)
 =
\lim_m \tau(a_m \si_{t + i \ga}(b_m)) \\
& =
\lim_m \tau(\si_t(b_m) a_m)
=
\tau(\si_t(b) a),
\end{align*}
when $\ga \in \La_{\be}$.

Let us begin with the following remark. Given $z=(z_1,\dots,z_n)$ consider the function $f_m(\zeta_1) = F_m(\zeta_1,z_2,\dots,z_n)$ as a function on $\bC$.
Then $f_m$ is analytic on $z_1$ since $F_m$ is analytic on $z$.
Therefore $f_m$ is analytic on
\[
D_1:=\{x + iy \in \bC \mid 0 \leq y \leq \be_1\}.
\]
Recall that by extending the action we get that
\[
\si_{(x+iy, z_2, \dots, z_n)}(b_m) = \si_{(x+iy,0,\dots,0)} \si_{(0,z_2,\dots,z_n)}(b_m).
\]
Moreover $f_m$ is continuous on $D_1$ and
\[
|f_m(x+iy)| = |\tau(a_m \si_{x + iy}'(b_m'))| \leq \nor{a_m} \cdot \nor{\si_{iy}'(b_m')},
\]
where $\si_{x+iy}'= \si_{(x+iy, 0,\dots, 0)}$ and $b_m' = \si_{(0,z_1,\dots,z_n)}(b_m)$.
However the function
\[
[0,\be_1] \ni y \mapsto \nor{\si_{iy}'(b_m')}
\]
is continuous, hence bounded.
Consequently $f_m(x+iy)$ is bounded.
By the Phragm\'{e}n-Lindel\"{o}f principle $f_m$ admits its supremum at the boundary of $D_1$, thus
\begin{align*}
|F_m(z_1, \dots, z_d)|
& =
|f_m(z_1)|
 \leq
\max\{ \sup_{x \in \bR} |f(x)|, \sup_{x \in \bR} |f(x+ i \be_1)| \} \\
& =
\max\{ \sup_{x \in \bR} |F(x,z_2,\dots, z_n)|, \sup_{x \in \bR} |F(x + i \be_1, z_2, \dots, z_n)| \}.
\end{align*}
The same holds for any co-ordinate.

For $\eps>0$ let $m_0 \in \bZ_+$ such that $\nor{a_l - a_m} < \eps/2\nor{a}$ and $\nor{b_l -b_m} < \eps/2 \nor{b}$ for $m,l \geq m_0$.
Following inductively the same arguments as above we derive that the function $F_l(z) - F_m(z)$ is bounded by the maximum of the values
\[
\sup_{x_1} \dots \sup_{x_n} |F_l((x_1,\dots, x_n) + i \ga) - F_m((x_1,\dots, x_n) + i \ga)|,
\]
with respect to $\ga = \sum \eps_k \be_k$ for all choices of $\eps_k=0,1$.
For $\ga \notin \La_{\be}$ we obtain
\begin{align*}
|F_l((x_1,\dots, x_n) + i \ga) - F_m((x_1,\dots, x_n) + i \ga)|
& = \\
& \hspace{-4cm} =
|\tau(a_l \si_{i\ga}\si_{x}(b_l)) - \tau(a_m \si_{i\ga}\si_{x}(b_m))| \\
& \hspace{-4cm} =
|\tau(a_l \si_{x}(b_l)) - \tau(a_m \si_{x}(b_m))| \\
& \hspace{-4cm} \leq
|\tau((a_l-a_m)\si_x(b_l))| + |\tau(a_m\si_x(b_l-b_m))| \\
& \hspace{-4cm} \leq
\nor{a_l - a_m} \nor{b_l} + \nor{a_m} \nor{b_l - b_m} < \eps .
\end{align*}
On the other hand for $\ga \in \La_{\un{\be}}$ we obtain
\begin{align*}
|F_l((x_1,\dots, x_n) + i \ga) - F_m((x_1,\dots, x_n) + i \ga)|
& = \\
& \hspace{-4cm} =
|\tau(a_l \si_{i\ga}\si_{x}(b_l)) - \tau(a_m \si_{i\ga}\si_{x}(b_m))| \\
& \hspace{-4cm} =
|\tau(\si_{x}(b_l)a_l) - \tau(\si_{x}(b_m) a_m)| \\
& \hspace{-4cm} \leq
|\tau(\si_x(b_l)(a_l-a_m))| + |\tau(\si_x(b_l-b_m)a_m)| \\
& \hspace{-4cm} \leq
\nor{a_l - a_m} \nor{b_l} + \nor{a_m} \nor{b_l - b_m} < \eps .
\end{align*}
Therefore $(F_m)$ is indeed a Cauchy sequence uniformly on $D$.

For the converse suppose that for every $a,b \in \A$ there exists a complex function $F_{a,b}$ which is analytic on the interior of
\[
D:=\{z \in \bC^n \mid 0 \leq \im(z_k) \leq \be_k, k=1, \dots, n\}
\]
and continuous on $D$ (hence bounded) such that $F_{a,b}(t) = \tau(a \si_t(b)$ and
\[
F(t+ i\ga) =
\begin{cases}
\tau(\si_{t}(b)a) & \text{ when } \ga \in \La_{\be}, \\
\tau(a \si_t(b)) & \text{ when } \ga \notin \La_{\be}.
\end{cases}
\]
For $a,b \in \A_{\text{an}}$ define $G_{a,b}(z) = \tau(a \si_z(b))$ with $z\in \bC^n$.
Then $G_{a,b}$ is entire analytic and
\[
G_{a,b}(t) = \tau(a \si_t(b)) = F_{a,b}(t) \foral t \in \bR^n.
\]
By the edge-of-the-wedge theorem for the positive cone $\bR_+^n$ we get that $F_{a,b}(z) = G_{a,b}(z)$ for all $z \in D$ hence $F_{a,b}(z) = \tau(a \si_z(b))$.
The second property of $F_{a,b}$ implies
\[
\tau(a \si_{i\ga}(b)) = F_{a,b}(i\ga) = \tau(ab)
\]
for $\ga \notin \La_{\be}$, and
\[
\tau(a \si_{i\ga}(b)) = F_{a,b}(i\ga) = \tau(ba)
\]
for $\ga \in \La_{\be}$. Hence $\tau$ satisfies the $(\si,\La_{\un{\be}})$-KMS condition.

As in \cite[Proposition 5.3.7]{BraRob97}, en passant we proved that such an $F_{a,b}$ satisfies
\[
\sup\{|F_{a,b}(z)| \mid z \in D \} \leq \nor{a} \cdot \nor{b},
\]
and that when $a,b \in \A_{\text{an}}$ then $F_{a,b}$ is the restriction of the function $z \mapsto \tau(a\si_z(b))$ to $D$.  \hfill{\qed}


\end{document}